\documentclass{amsart}

\usepackage{amssymb, array, verbatim, amscd}
\usepackage{amsmath, amsfonts,amsthm, color}
\usepackage[all]{xy}

\theoremstyle{plain}
\newtheorem{theorem}{Theorem}
\newtheorem{corollary}[theorem]{Corollary}
\newtheorem{lemma}[theorem]{Lemma}
\newtheorem{proposition}[theorem]{Proposition}

\theoremstyle{definition}

\newtheorem{definition}[theorem]{Definition}
\newtheorem{example}[theorem]{Example}

\newtheorem{remark}[theorem]{Remark}

\newcommand{\CU}{\mathcal{U}}

\newcommand{\CC}{{\mathcal C}}

\newcommand{\CI}{{\mathcal I}}

\newcommand{\N}{{\mathbb N}}

\newcommand{\Q}{{\mathbb Q}}
\newcommand{\Z}{{\mathbb Z}}
\newcommand{\id}{\mathrm{id}}
\newcommand{\boldP}{{\mathbf P}}

\newcommand{\CFP}{\mathcal{FP}}
\newcommand{\CFG}{\mathcal{FG}}

\newcommand{\Ima}{\mathrm{Im}}

\newcommand{\Ker}{\mathrm{Ker}}
\newcommand{\Coker}{\mathrm{Coker}}
\newcommand{\Pext}{\mathrm{Pext}}

\newcommand{\Gen}{\mathrm{Gen}}

\newcommand{\frakF}{\mathfrak{F}}

\newcommand{\Modr}{\mathrm{Mod}\text{-}}

\newcommand{\Dev}{\mathrm{Def}}

\newcommand{\Hom}{\operatorname{Hom}}

\newcommand{\Ext}{\operatorname{Ext}}

\begin{document}

\title[Defect functor and direct unions]{The defect functor of a homomorphism and direct unions}
\author{Simion Breaz} \thanks{Research supported by the
CNCS-UEFISCDI grant PN-II-RU-PCE-2012-4-0100}
\address{Babe\c s-Bolyai University, Faculty of Mathematics
and Computer Science, Str. Mihail Kog\u alniceanu 1, 400084
Cluj-Napoca, Romania}
\email{bodo@math.ubbcluj.ro}

\author{Jan \v Zemli\v cka}
\address{Department of Algebra, Charles University in Prague, Faculty of Mathematics and Physics Sokolovsk\' a 83, 186 75 Praha 8, Czech Republic}
\email{zemlicka@karlin.mff.cuni.cz}

\date{\today}


\thanks{}

\begin{abstract}
We will study commuting properties of the defect functor $\Dev_\beta=\Coker\Hom_\CC(\beta,-)$ 
associate to a homomorphism $\beta$ in a finitely presented category. As an application,
we characterize objects $M$ such that $\Ext^1_\CC(M,-)$
commutes with direct unions (i.e. direct limits of monomorphisms), assuming that $\CC$ has a generator 
which is a direct sum of finitely presented projective objects.
 \end{abstract}


\maketitle


\section{Introduction}\label{Intro}

Commuting properties of some canonical functors defined on some
 categories play important roles in the study of various
mathematical objects. For instance, finitely presented objects in a category with directed colimits are defined by
the condition that the induced covariant Hom-functor commuted with respect to all directed colimits. 
{In the case of 
module categories the equivalence between the property used in this definition
and the classical notion of finitely presented module was proved by Lenzing in \cite{Le}.} 
In that paper it is also proved that there are
strong connections between commuting properties of covariant Hom-functors and commuting properties of tensor product functors
with respect to direct products. These connections were extended to the associated derived functors in
\cite{BiH} and \cite{Brown}.  Moreover, Drinfeld proposed in
\cite{Drin} to use flat Mittag-Leffler modules in order to
construct a theory for infinite dimensional vector bundles. Recent progresses in this directions were obtained in \cite{BaS},
\cite{EAPT} and \cite{EAT}.
Auslander introduced in \cite{Aus} the class of coherent functors, and 
W. Crawley-Boevey characterized (in the case of module categories) these functors as those covariant functors which commutes 
with respect to direct limits and direct products, \cite[Lemma 1]{Cra}. 
This result was extended to locally finitely presented categories by H. Krause, \cite[Chapter 9]{Kr}. 
The influence of these functors is presented in \cite{Cra} and \cite{hart}.

Brown \cite{Brown} and Strebel \cite{Strebel} used
commuting properties of covariant $\Ext^1_\CC$-functors with respect to 
direct limits in order to characterize groups of type (FP). In
module theory an important ingredient used in the study of
tilting classes (e.g. \cite[Lemma 5.2.18 and Theorem 5.2.20]{GT06}) is a homological characterization, \cite[Theorem
4.5.6]{GT06}, of the closure $\underrightarrow{\lim}\CC$, where
$\CC$ is a class of $FP_2$-modules. This is based on the fact that
$\Ext_R^1(M,-)$ commutes with respect to direct limits whenever $M$
is an $FP_2$-module, \cite[Lemma 3.1.6]{GT06}. In the case of
Abelian groups, commuting properties of $\Ext^1$ functors with
respect particular direct limits were also studied in \cite{ABS}
and \cite{S11}.

In this paper we will focus on commuting properties with respect to 
direct limits for the defect functor associated to a homomorphism in a 
locally finitely presented abelian category. Let us introduce basic notions
which we will use in the sequel.
Let $M$ be an object in an additive category $\CC$ with directed colimits, and $G:\CC\to Ab$ a covariant functor. 
Furthermore suppose that $\frakF=(M_{i},
v_{ij})_{i,j\in I}$ is a directed system of objects in $\CC$ such that there exists $\underrightarrow{\lim}M_i$. 
and let $v_{i}:M_i\to \underrightarrow{\lim}M_i$ be the
canonical homomorphisms. {Then $(G(M_i), G(v_{ij}))$ is also a direct system, and we denote by 
$\underrightarrow{\lim}G(M_i)$ its direct limit.} Moreover,
we have a canonical homomorphism
\[\Gamma_\frakF:\underrightarrow{\lim}G(M_i)\to G(\underrightarrow{\lim} M_i)\] induced by the homomorphisms
$G(v_{i}):G(M_i)\to G(\underrightarrow{\lim} M_i)$, $i\in I$. 

We say that $G$ \textsl{commutes with respect to $\frakF$} if $\Gamma_\frakF$ is an isomorphism. 
The functor $G$ \textsl{commutes with respect to direct
limits} (\textsl{direct unions}, resp. \textsl{direct sums}) if the homomorphisms
$\Gamma_\frakF$ are isomorphisms for all directed systems $\frakF$
(such that all $v_{ij}$ are monomorphisms, resp. all direct sums).


Let $\CC$ be an additive category with direct limits. We recall
from \cite{AR} and \cite{ARV} that an object $M$ is \textsl{finitely
presented} (\textsl{finitely generated}) respectively if and only if $\Hom_\CC(M,-)$
commutes with respect to direct limits (of monomorphisms), i.e. the canonical
homomorphisms
\[
\Psi^{M}_\frakF:\underrightarrow{\lim}\Hom_\CC(M,M_i)\to
\Hom_\CC(M,\underrightarrow{\lim} M_i)\]
are isomorphisms for all direct systems $\frakF=(M_i, v_{ij})$ (such that all $v_{ij}$ are monomorphisms).
The category $\CC$ is \textsl{finite accessible} if $\CC$ has directed colimits and 
every object is a direct limit of finitely presented objects.
A cocomplete finitely accessible category $\CC$ is a \textsl{locally finitely presented category}. 


The notion of defect functor associated to a homomorphism 
extends the defect functor of an exact sequence used in \cite{Aus-Rei}. This functor represents generalizations 
for the following canonical functors: the $\Hom$-covariant functor induced by an object, the $\Pext$-covariant functor induced by an object, 
respectively the $\Ext^1$-covariant functor in the case when $\CC$ is a functor category. 

In Section \ref{Sect-defect} we introduce the defect functor $\Dev_\beta:\CC\to Ab$ associated to a homomorphism $\beta$, 
and we establish some basic properties for this functor. In Theorem~\ref{theorem-dev-restriction}  we show that the 
canonical decomposition of $\beta$ induced a short exact sequence of defect functors. 
Since $\Dev_\beta$ commutes with respect to direct products, we can apply \cite{Cra} and \cite{Kr-coll} to manage 
the case when $\Dev_\beta$ commutes with respect to all direct limits. 
Therefore we will focus our study to commuting 
properties with respect to particular direct limits.

In Section \ref{Sect-direct-limits} we study 
when the natural homomorphism 
$\Phi^{\beta}_\frakF:\underrightarrow{\lim}\Dev_\beta(M_i)\to \Dev_\beta(\underrightarrow{\lim}
M_i),$
where $\frakF=(M_{i}, v_{ij})_{i,j\in I}$ is a directed family in $\CC$, is an epimorphism. It is proved that 
$\Phi^{\beta}_\frakF$ is an epimorphism for all directed family $\frakF$ (of monomorphisms) if and only if $\beta$ is 
a section in the quotient category of $\CC$ modulo the ideal of all homomorphisms which factorizes through 
a finitely presented (generated) object. 
We apply these results in Sections \ref{Sect-direct-unions} and \ref{sect-dev-direct-sums} in order to 
characterize the homomorphisms $\beta$ such that 
$\Dev_\beta$ commutes with respect to direct unions or direct sums. Assuming that there is no $\omega$-measurable
cardinal we prove that it is enough to consider only commuting of $\Dev_\beta$
with respect to countable direct sums (Proposition~\ref{measurable}).

In Section \ref{Ext-unions}
(this section includes the results proved in the unpublished manuscript \cite{Br-uni}) 
we apply the previous results to characterize objects $M$ in a functor category with the property
that the functor $\Ext_\CC^1(M,-)$ commutes with respect to direct unions (Theorem \ref{main-th}).   
These are exactly the direct summands in direct sums of projective objects
and finitely presented objects. In \cite[Section 6]{Drin} the author used these objects (called, \textsl{2-almost
projective modules}) in order to study various kind of objects, e.g.
differentially nice $k$-schemes are defined using 2-almost
projective modules. These
objects are also studied in \cite{fin-r} for the case of
quasivarieties, cf. \cite[Proposition 4.3]{fin-r}.


For the case of coherent categories
these are exactly those objects such that the induced
$\Ext^1_\CC$-covariant functor commutes with respect to direct limits
(Corollary \ref{coherent}). We mention that in fact the structure
of these objects can be very complicated. For such an example we
refer to \cite[Lemma 4.3]{Pu}.

Furthermore, Theorem~\ref{thm2} gives a description of objects $M$ for which $\Ext_\CC^1(M,-)$ 
commutes with respect to direct sums using some splitting properties of
projective presentations of $M$.  We close the paper with a discussion about 
steadiness relative to $\Ext^1$, i.e. the condition when commuting of $\Ext^1(M,-)$ with respect to direct sums implies
commuting of $\Ext^1(M,-)$ with respect to direct unions.

In this paper $\CC$ will denote an \textsl{locally finitely presented abelian category,} 
i.e.
$\CC$ is a Grothendieck category with a generating set of finitely presented objects. Therefore, an object 
is finitely generated iff it is an epimorphic image of a finitely presented object \cite[Proposition 1.69]{AR}, 
and the structural homomorphisms associated to direct unions are monomorphisms by \cite[Proposition 1.62]{AR}.

\section{The defect functor associated to a homomorphism}\label{Sect-defect}

In order to define the defect functor $\Dev_\beta$ associated to a homomorphism $\beta$ 
it is useful to consider, as in \cite{Kr}, 
the big category $(\CC, Ab)$ of all additive covariant functors from $\CC$ into the category of all abelian groups. 
Albeit $(\CC,Ab)$ is not a category we can construct pointwise all notions which define abelian categories 
(kernels, cokernels, direct sums etc.), and the universal properties associated to these notions can be 
transfered from $\CC$ to $(\CC,Ab)$. For instance, if $\eta:F\to G$ is a natural transformation then we can define a functor 
$\Coker(\eta)$ and a natural transformation $\mu:G\to \Coker(\eta)$ in the following way: For all $X\in \CC$ we define $\Coker(\eta)(X)=\Coker(\eta_X)=G(X)/\Ima(\eta_X)$, and for every $\alpha:X\to Y$ we define $\Coker(\eta)(\alpha): \Coker(\eta)(X)\to \Coker(\eta)(Y)$ is the unique map which make the diagram 
\[\xymatrix{          
  F(X)\ar[d]^{F(\alpha)} \ar[r]^{\eta_X}& G(X)\ar[rr]^{\mu_X}\ar[d]^{G(\alpha)} && 
  G(X)/\Ima(\eta_X)\ar@{-->}[d]^{\Coker(\eta)(\alpha)}\ar[r] & 0\\
    F(Y) \ar[r]^{\eta_Y}& G(Y)\ar[rr]^{\mu_Y} && G(Y)/\Ima(\eta_Y)\ar[r] & 0\\ }\] commutative, where
$\mu_X:G(X)\to \Coker(\eta)(X)$ and $\mu_Y:G(Y)\to \Coker(\eta)(Y)$ are the canonical epimorphisms. 
It is not hard to see that $\Coker(\eta)$ is a functor, and the collection 
$\mu_X$ define a natural transformation $G\to \Coker(\eta)$ which has the 
same universal property as those which defines the classical cokernel in additive categories.    


\begin{definition} Suppose that $\beta : L\to P$ is a homomorphism in $\CC$. Then $\beta$ induces a natural transformation 
$\beta^*:\Hom(P,-)\to \Hom(L,-)$. 
The functor $$\Dev_\beta(-)=\Coker(\Hom(\beta,-))$$ will be called 
\textsl{the defect functor associated to $\beta$}.
\end{definition}

It is clear from the previous observation that $\Dev_\beta$ is characterized by the conditions:
\begin{enumerate}
\item[(i)] $\Dev_\beta(X)=\Hom(L,X)/\Ima(\Hom(\beta,X))$ for each object $X$ and 
\item[(ii)] $\Dev_\beta(\gamma)(\alpha+B_X)=\gamma\alpha+B_Y$ for each objects $X,Y$ and homomorphisms $\gamma\in \Hom(X,Y)$, $\alpha\in\Hom(K,X)$
where $B_X=\Ima(\Hom(\beta,X))$ and $B_Y=\Ima(\Hom(\beta,Y))$.
\end{enumerate}

In fact, if $f:X\to Y$ is a homomorphism then we have a commutative diagram: 
\begin{equation*}
\label{basic-diagram}
 \begin{CD}
\Hom(P,X)@>>> \Hom(L,X)@>>>\Dev_\beta(X)@>>>0\\
@VVV @VVV@VVV@.@.\\
\Hom(P,Y)@>>> \Hom(L,Y)@>>>\Dev_\beta(Y)@>>>0. 
\end{CD}
\end{equation*}

Here are some examples: 

\begin{example}\label{dev} Let $\beta : L\to P$ be a homomorphism in $\CC$.
\begin{enumerate}
\item If $\CC$ is abelian, $P$ is projective and $\beta$ a monomorphism, then $\Dev_\beta(-)$ is canonically equivalent 
to $\Ext^1(P/\beta(L),-)$.
\item If $P=0$, then $\Dev_\beta(-)$ is canonically equivalent to $\Hom(L,-)$.
\item If $\beta$ is an epimorphism and $\upsilon:K\to L$ is the kernel of $\beta$ then $\Dev_\beta(-)$ represents the 
covariant defect functor associated to the exact sequence $0\to K\overset{\upsilon} \to L\overset{\beta}\to P\to 0$, 
\cite{Krause-def}. 
\item If $R$ is a unital ring, $\CC=\Modr R$, and $L$ and $P$ are finitely generated and projective 
then $\Dev_\beta(R)$ represents the transpose of $P/\beta(L)$.
\end{enumerate}
\end{example}



In the following we will prove some general properties of defect functors.
Since in the category of all abelian groups the direct products are exact, it is easy to see that $\Dev_\beta$ commutes 
with respect to direct products. Moreover, in many situations the study of these functors can be reduced to the study of defect functors 
associated to monomorphisms or to epimorphisms.

\begin{theorem}\label{theorem-dev-restriction}
Let $\beta:L\to P$ be a homomorphism in the abelian category $\CC$. If $i_K:K\to L$ is the kernel of $\beta$, $\pi_K:L\to L/K$ is the canonical
epimorphism, and $\overline{\beta}:L/K\to P$ is 
the homomorphism induced by $\beta$ then there exists a canonical exact sequence of functors and natural transformations 
$$0\to \Dev_{\overline{\beta}}\to \Dev_\beta\to \Dev_{\pi_K}\to 0.$$
\end{theorem}

\begin{proof}
Starting with the exact sequence 
$$ 0\to K\overset{\iota_K} \to L\overset{\beta}\to P\to M\to 0,$$ where $M$ is the cokernel of $\beta$, we obtain 
the short exact sequences
$$0\to K\overset{\iota_K} \to L\overset{\pi_K}\to  L/K\to 0$$ and 
$$0\to L/K\overset{\overline{\beta}_K}\to  L\to M\to 0.$$ 

Passing to the Hom covariant functors induced by the objects involved in the previous exact sequences we obtain, 
using the Ker-Coker Lemma, the 
following commutative diagram of functors and natural transformations:
\[\xymatrix{          
&  & & 0\ar[d] &0\ar[d]  &\\
0\ar[r] & (M,-) \ar@{=}[d] \ar[r] & (P,-) \ar@{=}[d] \ar[r]^{\overline{\beta}^*} & (L/K,-)\ar[r] \ar[d] &
\Dev_{\overline{\beta}}(-) \ar[r] \ar[d]	&  0 \\ 
0\ar[r] & (M,-)  \ar[r] & (P,-)  \ar[r]^{\beta^*} & (L,-) \ar[r] \ar[d]	& \Dev_\beta(-) \ar[r] \ar[d] & 0\\ 
 &     &  & \Dev_{\pi_K}\ar@{=}[r]\ar[d]&\Dev_{\pi_K} \ar[d]   \\
		& &  & 0& 0 & ,\\
	}\]
hence the statement of the theorem is proved.
\end{proof}

\begin{proposition}\label{half-exact}
If $\beta:L\to P$ is a homomorphism in $\CC$, the following statements are true:
\begin{enumerate}
 \item If $P$ is projective and $M=\Coker(\beta)$ then every exact sequence $$0\to X\to Y\to Z\to 0$$ 
 induces an exact sequence 
 \[0\to (M,X)\to (M,Y)\to (M,Z)\to \Dev_\beta(X)\to \Dev_\beta(Y)\to \Dev_\beta(Z).\]

  \item If $L$ is projective then $\Dev_\beta$ preserves the epimorphisms.

    \item If $L$ and $P$ are projective then $\Dev_\beta$ is a right exact functor.
 \end{enumerate}
\end{proposition}

\begin{proof}
Let $0\to X\to Y\to Z\to 0$ be an exact sequence. Applying the Hom-functors we obtain the following commutative diagram

\[\xymatrix{          
& 0\ar[d] & 0\ar[d] & 0\ar[d] & \\
0\ar[r] & (M,X)\ar[d]\ar[r] & (M,Y)\ar[d]\ar[r] & (M,Z)\ar[d] & \\
0\ar[r] & (P,X)\ar[d]\ar[r] & (P,Y)\ar[d]\ar[r] & (P,Z)\ar[d] & \\
0\ar[r] & (L,X)\ar[d]\ar[r] & (L,Y)\ar[d]\ar[r] & (L,Z)\ar[d] & \\
& \Dev_\beta(X)\ar[d]\ar[r] & \Dev_\beta(Y)\ar[d]\ar[r] & \Dev_\beta(Z)\ar[d] & \\
& 0 & 0 & 0 &, \\
}
\]
and the statements are obvious.
\end{proof}

\begin{remark}
Recently the defect functor associated to a homomorphism between projective object was involved in the study of silting
modules, \cite{Ange-silting}: a homomorphism $\beta:L\to P$ with $L$ and $P$ projective objects is a \textsl{silting
module} if $\Gen(P/\beta(L))=\Ker(\Dev_\beta)$. 
\end{remark}

\section{The defect functor and direct limits}\label{Sect-direct-limits}

Throughout the section we suppose that $L\overset{\beta}\to P\overset{\alpha}\to M\to 0$ is an 
exact sequence in $\CC$, $\frakF=(M_{i}, v_{ij})_{i,j\in I}$
is a direct system of objects in $\CC$, and $v_{i}:M_i\to \underrightarrow{\lim}M_i$ are the
canonical homomorphisms. Furthermore, we denote by
$$\Phi^{\beta}_\frakF:\underrightarrow{\lim}\Dev_\beta(M_i)\to \Dev_\beta(\underrightarrow{\lim}
M_i)$$ the natural homomorphisms induced by the families
$\Dev_\beta(v_{ij})$, $i,j\in I$, and
$\Dev_\beta(v_{i})$, $i\in I$. 
{Following the general definition considered in Section \ref{Intro}, 
we say that $\Dev_\beta(-)$ \textsl{commutes with respect to $\frakF$} if $\Phi^\beta_\frakF$ is an isomorphism. 
The functor $\Dev_\beta(-)$ \textsl{commutes with respect to direct
limits} (\textsl{direct unions}, resp. \textsl{direct sums}) if the homomorphisms
$\Phi^\beta_\frakF$ are isomorphisms for all directed systems $\frakF$
(such that all $v_i$ are monomorphisms, resp. all direct sums).}

We have the following useful commutative diagram

\begin{equation*}
\xymatrix{
0\ar[r] & \underrightarrow{\lim}(M,M_i) \ar[r] \ar[d]^{\Psi^M_\frakF} &
\underrightarrow{\lim}(P,M_i) \ar[r] \ar[d]^{\Psi^P_\frakF} &  \underrightarrow{\lim}(L,M_i)
\ar[d]^{\Psi^L_\frakF} \ar[r]^{\underrightarrow{\lim}\xi_i} & 
\underrightarrow{\lim}\Dev_\beta(M_i) \ar[r]\ar[d]^{\Phi^\beta_\frakF} & 0 \\
0\ar[r] & (M,\underrightarrow{\lim}M_i) \ar[r] &
(P,\underrightarrow{\lim}M_i) \ar[r]^{\beta^*} &
(L,\underrightarrow{\lim}M_i) \ar[r]^{\xi} & \Dev_\beta(\underrightarrow{\lim}M_i) \ar[r] & 0 }
\tag{D1} \label{D10}
\end{equation*}
whose rows are exact, where the natural homomorphisms $\Psi_\frakF^X$ are defined in Section \ref{Intro}.

{

Using this diagram we have the following simple consequences:

\begin{corollary} \label{cor-L=fg} \label{cor-L=fp}
\begin{enumerate} 
 \item 
If $L$ is finitely presented and $\beta:L\to P$ is a homomorphism, then for every direct family $\frakF$
the canonical homomorphism $\Phi_\frakF^\beta$ is an epimorphism.

\item If $L$ is finitely generated and $\beta:L\to P$ is a homomorphism, then for every direct family of monomorphisms $\frakF$
the canonical homomorphism $\Phi_\frakF^\beta$ is an epimorphism.
\end{enumerate}
\end{corollary}

\begin{example}\label{ex-fg-non}
There exists a homomorphism $\beta:L\to P$ and a direct family $\frakF$ (of monomorphisms) such that $L$ is finitely presented
(generated) 
and the canonical homomorphism $\Phi_\frakF^\beta$ is not an isomorphism. 
\end{example}

\begin{proof}
Let $\CC$ be the category of all abelian groups. 
If $p$ is a prime number we denote by $\Z_p=\{\frac{m}{p^k}\mid m\in\Z,\ k\in \N\}\leq \Q$. 
If $\beta:\Z\to \Z_p$ is the canonical inclusion in the category of all abelian groups then for every torsion-free 
abelian group $A$ we have a natural isomorphism $$\Dev_\beta(A)\cong A/D_p(A),$$ 
where $D_p(A)$ is the maximal $p$-divisible subgroup of $A$. 

We can write the abelian group $\Q$ as a union of a chain of cyclic subgroups $F_n=\frac{1}{n!}\Z$, $n\in\N^*$, 
where the connecting  homomorphisms $u_{m,n}:F_m\to F_n$, $m<n$, are the inclusion maps. Since $\Hom(\Z_p,F_n)=0$ for
all $n>0$, it follows that we can identify $\Dev_\beta(F_n)=F_n$ and $\Dev_\beta(u_{m,n})=u_{m,n}$ for all $m,n\in\N^*$.
Then $\underrightarrow{\lim}\Dev_\beta(F_n)=\Q$. But $\Dev_\beta(\underrightarrow{\lim}F_n)=\Dev_\beta(\Q)=0$, hence 
$\Phi_\frakF^\beta:\Q\to 0$ is not a monomorphism.  
\end{proof}

We will use the following lemma:

\begin{lemma} An object $M$ is
finitely generated if and only if there exists an exact sequence $0\to
L\to P\to M\to 0$ with $P$ a finitely presented object. 

Consequently, if $M$ is finitely generated then for every direct system $\frakF$ 
the natural homomorphism $\Psi_\frakF^M$ is a monomorphism.
Moreover,
$M$ is finitely presented if and only if $L$ is finitely generated.
\end{lemma}

\begin{proof}
The first part is proved in \cite[Proposition 1.69]{AR}, while for the other statements
we can apply Ker-Coker Lemma on diagram \eqref{D10}.
\end{proof}
}

{Applying the above definitions to the diagram \eqref{D10}, it is not hard
approach that case when $\Hom(P,-)$ commutes with respect to direct sums, direct unions, respectively direct limits.
We recall that $P$ is called \textsl{small} if $\Hom(P,-)$ commutes with respect to direct sums. }

\begin{proposition}\label{basic-def}
Let $\beta:L\to P$ be a homomorphism. 
\begin{enumerate} 
\item
{Suppose that $P$ is a small object. The functor $\Dev_\beta$ commutes with respect to direct sums if and only
if $L$ is a small object.}

\item
Suppose that $P$ is a finitely generated object. The functor $\Dev_\beta$ commutes with respect to direct unions if and only
if $L$ is finitely generated.

\item Suppose that $P$ is a finitely presented object. Then $\Dev_\beta$ commutes with respect to direct limits if and
only if $L$ is finitely presented.
\end{enumerate}
\end{proposition}

\begin{proof}
{(1) Let $\frakF=(M_{i})_{i\in I}$ be a family of objects in $\CC$. We construct
a diagram \eqref{D10} induced by the direct sum of $\frakF$. Since the class of small objects is closed with 
respect to epimorphic images,
$\Psi^M_\frakF$ and $\Psi^P_\frakF$ are isomorphisms. Therefore
$\Phi^\beta_\frakF$ is an isomorphism if and only if $\Psi^L_\frakF$
is an isomorphism. The conclusion is now obvious.

(2) The proof follows the same steps as for (1), using this time a direct system 
$\frakF=(M_{i}, \nu_{ij})_{i,j\in I}$ such that all $\nu_{ij}$ are monomorphisms, and the fact that the class of 
finitely generated objects is closed with respect to epimorphic images.
}

(3) Suppose that $\Dev_\beta$ commutes with respect to direct
limits. By what we just proved $L$ is finitely generated, hence $M$ is finitely presented.
Therefore, for every direct system $\frakF=(M_{i},
\nu_{ij})_{i,j\in I}$ the homomorphisms $\Psi^M_\frakF$ and
$\Psi^P_\frakF$ are isomorphisms. Therefore $\Psi^L_\frakF$ is an
isomorphism, hence $L$ is finitely presented.

Conversely, the objects $L$, $M$, and $P$ are finitely presented, hence the first three vertical maps in diagram (\ref{D10})
are isomorphisms. Then $\Phi_\frakF^\beta$ is also an isomorphism.
\end{proof}

Using the statement (2) in the above proposition we can reformulate the characterization presented in \cite[Lemma 1]{Cra}
for the case of direct unions. Since the proof is \textsl{verbatim} to Crawley-Boevey's proof, it is omitted.

\begin{theorem}\label{commuting-functors}
A functor $F:\CC\to Ab$ commutes with respect direct products and direct unions if and only if it is naturally 
isomorphic to a defect functor $\Dev_\beta$ associated to a homomorphism $\beta:L\to P$ with $L$ and $P$ finitely 
generated.
\end{theorem}

{Using the same techniques as in \cite{Br-ext-lim}, it is not hard to see that when 
$L$ and $P$ are projective the three commuting properties considered in the Proposition \ref{basic-def}
are equivalent. 

\begin{proposition}\label{equiv-lim}
Let $\beta:L\to P$ be a homomorphism between projective right $R$-modules. Then the following are equivalent: 
\begin{enumerate} 
\item  $\Dev_\beta$ commutes with respect to direct limits;
\item
$\Dev_\beta$ commutes with respect to direct unions;
\item
$\Dev_\beta$ commutes with respect to direct sums;

\item $\Dev_\beta$ commutes with respect to direct sums of copies of $R$;
\end{enumerate}
Under these conditions $\Dev_\beta(R)$ is a finitely presented left $R$-module.
 \end{proposition}
 
 \begin{proof}
(4)$\Rightarrow$(1) From Proposition \ref{half-exact} and from the proof of Watts's theorem \cite[Theorem 1]{Watts}, 
we obtain that
$\Dev_\beta(-)$ is naturally isomorphic 
to $-\otimes_R\Dev_\beta(R)$. Therefore it preserves direct limits. 

Moreover, 
if these equivalent conditions are satisfied the functor $-\otimes_R\Dev_\beta(R)$ preserves direct products. 
This is true exactly if
the left $R$-module $\Dev_\beta(R)$ is finitely presented.
 \end{proof}

It is well known that if $G:\CC\to Ab$ is an additive functor then for every family $\frakF=(M_i)_{i\in I}$ then natural
homomorphism $\oplus_{i\in I}G(M_i)\to G(\oplus_{i\in I}M_i)$ is a monomorphism. Therefore in the above proposition it is 
enough to verify if the natural homomorphisms $\Phi_\frakF^\beta$ are epimorphisms.
}

In the following we will study the case when the natural homomorphisms $\Phi_\frakF^\beta$ are epimorphisms.

\begin{lemma}\label{epi-dev}
Let $\beta:L\to P$ be a homomorphism, $\frakF=(M_{i}, v_{ij})_{i,j\in I}$ a direct system,
and let $f:L\to \underrightarrow{\lim}M_i$ be a homomorphism. Using the same notations as in diagram \eqref{D10}, 
the following are equivalent:
\begin{enumerate}
 \item $\xi(f)\in \Ima(\Phi^\beta_\frakF)$;
 \item there exists $k\in I$, $h:L\to M_k$ and $g:P\to \underrightarrow{\lim}M_i$ 
	such that $f=g \beta+v_k h$.
 \end{enumerate}
\end{lemma}

\begin{proof} The homomorphisms from (2) can be represented in the following diagram
\[\xymatrix{
 & L\ar[r] ^\beta \ar@{-->}[dl]_h \ar[d]^{f} & P\ar@{-->}[dl]^g \\
 M_k \ar[r]_{v_k} &\underrightarrow{\lim}M_i & .
} 
\]

(1)$\Rightarrow$ (2) 
If we look at the commutative diagram \eqref{D10}, 
we observe that $\xi(f)\in \Ima(\Phi^\beta_\frakF)$ if and only if there 
is an element $x\in
\underrightarrow{\lim}\Hom_\CC(L,M_i)$ such that $\xi(f)=\Phi_\frakF^\beta
(\underrightarrow{\lim}\xi_i)(x)=\xi\Psi_\frakF^L(x)$. 
Then $f-\Psi_\frakF^L(x)=\beta^*(g)=g\beta$ for some element
$g\in \Hom_\CC(P,\underrightarrow{\lim}M_i)$.

Since $x\in \underrightarrow{\lim}\Hom_\CC(L,M_i)$, there exist
$k\in I$ and $h\in \Hom_\CC(L,M_k)$ such that
$x=\overline{v}_k(h)\in \Ima\overline{v}_k$, where
$\overline{v}_k:\Hom(L,M_k)\to
\underrightarrow{\lim}\Hom_\CC(L,M_i)$ is the structural
homomorphism associated to the direct limit. Since
$\Psi_\frakF^\beta\overline{v}_k=\Hom_\CC(L,v_k)$, it follows that
$\Psi_\frakF^\beta(x)=\Hom_\CC(L,v_k)(h)=v_k h$.
Thus $f=g\beta+v_k h$.

(2)$\Rightarrow $(1) If $f=g \beta+v_k h$ then $f-v_k h\in \Ima(\beta^*)$, hence 
$$\xi(f)=f+\Ima(\beta^*)=v_i h+\Ima(\beta^*)=
\xi\Psi_\frakF^L(\overline{v}_k(h))=\Phi_\frakF (\underrightarrow{\lim}\xi_i)(\overline{v}_k(h)),$$
and the proof is complete.
\end{proof}

In the following  $\mathcal{FP}$ will be the ideal in $\CC$ of those homomorphisms which factorize through a 
finitely presented object, {i.e. $\mathcal{FP}$ represents the collection of subgroups $\mathcal{FP}(A,B)\leq\Hom_\CC(A,B)$,
$A,B\in\CC$,
of those homomorphisms $A\to B$ which factorize through a finitely presented object.} Then $\CC/\CFP$ will denote the quotient category which has as objects the same objects as $\CC$
and $$\Hom_{\CC/\CFP}(A,B)=\Hom_\CC(A,B)/\CFP(A,B).$$ 

{In the following assertion, if $f:A\to B$ and $h:A\to C$ are homomorphisms, we will denote by $(f,h)^t:A\to B\oplus C$ the 
canonical homomorphism  induced by $f$ and $h$.}

\begin{theorem}\label{comm-limits}
Let $\beta:L\to P$ be a homomorphism. The following are equivalent:
\begin{enumerate}
 \item for every direct system $\frakF$ the map $\Phi^\beta_\frakF$ is an epimorphism;
 \item there exists $g:P\to L$ such that $1_L-g\beta$ factorizes through a finitely presented object;
 \item the induced homomorphism $\overline{\beta}$ in $\CC/\CFP$ is a section;
\item there exists a homomorphism $h:L\to F$ such that $F$ is a finitely presented object 
 and the induced map $(f,h)^t:L\to P\oplus F$ is a splitting monomorphism.
 \end{enumerate}
\end{theorem}

\begin{proof}
 (1)$\Rightarrow$(2) We can write $L$ as a direct limit of finitely presented objects, $L=\underrightarrow{\lim}L_i$. 
 Then an application of Lemma \ref{epi-dev} for $f=1_L$ gives us the conclusion.
 
  (2)$\Rightarrow$(1) Since $1_L-g\beta$ factorizes through a finitely presented object, there exists a finitely presented
  object $F$ and two homomorphisms $h_1:L\to F$, $h_2:F\to L$ such that $1-g\beta=h_2h_1.$ 
  
  Since $F$ is finitely presented, for every direct limit $\underrightarrow{\lim}M_i$ and every homomorphism 
  $f:L\to \underrightarrow{\lim}M_i$ we can find an index $i$ and a homomorphism $f_i:F\to M_i$ such that $fh_2=v_i f_i$. 
  It follows that $f(1-g\beta)=fh_2h_1=v_i f_i h_1$. Then there exists $g'=fg:P\to \underrightarrow{\lim}M_i$ and 
  $h=f_i h_1:L\to M_i$ such that $f=g'\beta+v_ih$, and we apply Lemma \ref{epi-dev} to complete the proof.  
 
 (2)$\Leftrightarrow$(3) This is obvious. 
 
 
 (2)$\Rightarrow$(4) Let $g$ be as in (2) and $h=1_L-g\beta$. There exists a finitely presented object $F$ and two maps 
 $h_1:L\to F$, $h_2:F\to L$ such that 
 $h=h_2 h_1$. Then the map $(\beta,h_1)^t:L\to P\oplus F$ induced by  $\beta$ and $h_1$ is a splitting monomorphism, and 
 a left inverse is the homomorphism $(g,h_2):P\oplus F\to L$ induced by $g$ and $h_2$.
 
 (4)$\Rightarrow$(2) Let $g':P\oplus F\to L$ be a left inverse for $(\beta,h)^t$. Then $1_L=g_{|P}\beta+g_{|F}h$, hence
 $1_L-g_{|P}\beta$ factorizes through a finitely presented object.
\end{proof}





\section{Commuting with direct unions}\label{Sect-direct-unions}

Recall from \cite[Proposition 1.62]{AR} that in our hypotheses the
structural maps $v_i:M_i\to \underrightarrow{\lim}M_i$ of a direct union are monomorphisms.

Since the class of finitely generated objects is closed with respect to epimorphic images, 
we will prove that Theorem \ref{comm-limits} can be improved to characterize the commuting of 
$\Dev_\beta$ with respect to direct unions. 

\begin{theorem}\label{unions-1}
Let $\beta:L\to P$ be a homomorphism. The following are equivalent:
\begin{enumerate}
 \item for every direct system $\frakF$ of monomorphisms the induced homomorphism $\Phi_\frakF^\beta$ is an epimorphism;
 
 \item there exists $g:P\to L$ such that $1_L-g\beta$ factorizes through a finitely generated object;

 \item if $\CFG$ is the ideal of all homomorphisms which factorize through a finitely generated object then 
 the induced homomorphism $\overline{\beta}$ in $\CC/\CFG$ is a retract;
 \item there exists a homomorphism $h:L\to M$ such that $h$ factorizes through a finitely generated object 
 and the induced map $(\beta,h)^t:L\to P\oplus M$ is a splitting monomorphism.
\end{enumerate}

 
 \begin{enumerate}
 \item[(5)] there exists a homomorphism $h:L\to F$ such that $F$ is a finitely generated object 
 and the induced map $(f,h)^t:L\to P\oplus F$ is a splitting monomorphism,
 \item[(6)] there exists a finitely generated subobject $H\leq L$ such that the induced homomorphism 
$\overline{\beta}:L/H\to P/\beta(H)$ is a split mono and there exists a left inverse for $\overline{\beta}$ which can be lifted 
to a homomorphism $P\to L$. 
 \end{enumerate}
\end{theorem}

\begin{proof}
It is enough to prove the equivalence (2)$\Leftrightarrow$(6) since for the other equivalences we can repeat the 
arguments of the proof of Theorem \ref{comm-limits}, using the fact that $L$ can be written as a direct union of 
its finitely generated subobjects.

(2)$\Rightarrow$(6) 
By (2) we know that there exists a homomorphism $g:P\to L$ such that 
$1_L-g\beta$ factorizes through a finitely generated object. Therefore there exists a subobject $H\leq L$ such that 
$\Ima(1_L-g\beta)\subseteq H$. If $h:H\to L$ is the embedding of $H$ in $L$ then there exists a homomorphism $\gamma:L\to H$ such that 
$1_L-g\beta= h\gamma $.

Since $h=g\beta h+h \gamma h$, we have $\Ima(g\beta h)\leq
\Ima(h)$, hence $g\beta(H)\leq H$. Therefore there are
canonical homomorphisms $\overline{\beta}:L/H\to P/\beta(H)$
and $\overline{g}:P/\beta(H)\to K/H$ which are induced by $\beta$, respectively $g$, and the diagram
\begin{equation*}
\begin{CD}
L@>{\beta}>> P@>{g}>> L \\ @V{\pi_H}VV @V{\pi_{\beta(H)}}VV @V{\pi_H}VV \\
L/H@>{\overline{\beta}}>> P/\beta(H)@>{\overline{g}}>> L/H
\end{CD}
\end{equation*}
 is commutative, where the vertical arrows are the
canonical epimorphisms.

Moreover, $\pi_H h=0$, hence
$\overline{g}\overline{\beta}\pi_H=\pi_H g\beta=\pi_H(1_L-h
\gamma)=\pi_H$. Since $\pi_H$ is an epimorphism we have
$\overline{y}\overline{\beta}=1_{K/K_k}$, hence
$\overline{\beta}$ is a splitting monomorphism. 

(6)$\Rightarrow$(2) Let $g:P\to L$ be a homomorphism such that $g(\beta(H))\subseteq H$ and the induced homomorphism
$\overline{g}:P/\beta(H)\to L/H$ satisfies the equality $\overline{g}\overline{\beta}=1_{L/H}$. Then 
$\Ima(1_L-g\beta)\subseteq H$, and the proof is complete.
\end{proof}

If $P$ is projective the lifting condition stated in (6) is always satisfied. This is not the case if $P$ is not projective. 

\begin{example}\label{Ex-cond(6)}
Let $\CC$ be the category of all abelian groups, and let $\Z_p$ be the subgroup of $\Q$ defined in Example \ref{ex-fg-non}.

If $\beta:\Z_p\to \Q$ is the inclusion map then the induced homomorphism $\overline{\beta}:\Z_p/\Z\to \Q/\Z$ is split mono. 
But $\Hom(\Q,\Z_p)=0$, so the left inverse of $\overline{\beta}$ (in this case this left inverse is unique) 
cannot be lifted to a homomorphism $\Q\to \Z_p$.   
\end{example}


{
We obtain the following interesting characterization of pure-projective objects. Let us recall that
an exact sequence $0\to A\to B\to C\to 0$ is pure (and $B\to C$ is a pure epimorphism) 
if all finitely presented objects are projective with respect to it, and an 
object is projective if and only if it is projective with respect to all pure exact sequences. It is not hard to
see that an object is pure-projective iff it is a direct summand of a direct sum of finitely presented objects. 
As in the standard homological
algebra we can define the functor $\mathrm{Pext}^1(M,-)$ as $\Dev_\beta$, 
where $\beta:P\to M$ is a pure epimorphism such that 
$P$ is pure-projective. Remark that $M$ is pure-projective iff $\mathrm{Pext}^1(M,-)=0$. For more details we refer to 
\cite[Appendix A]{Jen-Len}.

\begin{proposition}\label{pure-proj}
The following are equivalent for an object $M\in\CC$: 
\begin{enumerate} 
\item  The functor $\mathrm{Pext}^1(M,-)$ commutes with respect to direct limits;
\item
$\mathrm{Pext}^1(M,-)$ commutes with respect to direct unions;
\item
$M$ is pure projective.
\end{enumerate}
 \end{proposition}

\begin{proof}
It is enough to prove that (2)$\Rightarrow$(3).  

Let $M$ be an object such that $\mathrm{Pext}^1(M,-)$ commutes with respect to direct unions. Since $M$ is a direct 
limit of finitely presented objects, there exists a pure exact sequence 
$$0\to L\overset{\beta} \to \oplus_{i\in I}P_i\to M\to 0$$ such that
all $P_i$ are finitely presented objects. Hence $\mathrm{Pext}^1(M,-)=\Dev_\beta$, and we apply 
Theorem \ref{unions-1}. Therefore there exists a finitely generated subobject $K\leq L$ such that the induced
homomorphism $\overline{\beta}:L/K\to \oplus_{i\in I}P_i/\beta(K)$ is a splitting monomorphism. But 
$\Coker(\overline{\beta})\cong M$, hence $M$ is isomorphic to a direct summand of $\oplus_{i\in I}P_i/\beta(K)$. Since 
$\beta(K)$ is finitely generated we can view $\beta(K)$ as a subobject of a subsum $\oplus_{i\in J}P_i/\beta(K)$, 
were $J$ is a finite subset of $I$. Since $\oplus_{i\in J}P_i$ is finitely presented, it follows that 
$\oplus_{i\in J}P_i/\beta(K)$ is also finitely presented, hence $M$ is a direct summand of a direct sum of finitely presented
objects. Then $M$ is pure-projective.
\end{proof}
}

The next observation allows us to prove that, in order to study the commuting properties with 
respect to direct unions, it is enough to restrict to defect functors associated to the homomorphisms which appear 
in the canonical decomposition of $\beta$.

\begin{proposition}\label{monomorphism-unions}
Let $\beta:L\to P$ be an epimorphism. Then for every direct system $\frakF=(M_{i}, v_{ij})_{i,j\in I}$ of monomorphisms, 
 the canonical map $\Phi_\frakF^\beta$ is a monomorphism. 
\end{proposition}

\begin{proof}
Let $x\in \Ker(\Phi_\frakF^\beta)$. Then there exists $y\in \underrightarrow{\lim}\Hom(L,M_i)$ such that 
$x=\underrightarrow{\lim}\xi_i(y)$ and $\Psi^L_\frakF(y)$ factorizes through $\beta$. 
There exists $i\in I$ and $\alpha_i:L\to M_i$ such that $y=\overline{v}_i(\alpha_i)$, where $\overline{v}_i$ denotes the 
canonical map $\overline{v}_i:\Hom(L,M_i)\to \underrightarrow{\lim}\Hom(L,M_i)$. 
Then $\Psi_\frakF^\beta\overline{v}_i(\alpha_i)=v_i\alpha_i$ factorizes through $\beta$. Let 
$\gamma:P\to \underrightarrow{\lim}M_i$ be a homomorphism such that $v_i\alpha_i=\gamma\beta$. 

Let $K=\Ker(\beta)$ and $\iota_K:K\to L$ be the canonical homomorphism. Then $v_i\alpha_i\iota_K=\gamma\beta\iota_K=0$. 
Since the structural homomorphisms $v_i$ are monomorphisms we obtain $\alpha_i\iota_K=0$, hence $\alpha_i$ factorizes 
through $\beta$. Then $x=0$, and the proof is complete.  
\end{proof}

\begin{corollary}
Let $\beta:L\to P$ be an epimorphism. The following are equivalent:
\begin{enumerate}
 \item the functor $\Dev_\beta$ commutes with respect to direct unions;
 
 \item for every direct family $\frakF$ of monomorphisms the induced homomorphism $\Phi_\frakF^\beta$ is an epimorphism;
\end{enumerate} 
\end{corollary}

Using Theorem \ref{theorem-dev-restriction} and Proposition \ref{monomorphism-unions} we have the following result:

\begin{theorem}\label{cor-theorem-dev-restriction}
Suppose that $\beta:L\to P$ is a homomorphism in the abelian category $\CC$, $i_K:K\to L$ is the kernel of $\beta$, $\pi_K:L\to L/K$ is the canonical
epimorphism, and $\overline{\beta}:L/K\to P$ is 
the homomorphism induced by $\beta$. Then 
$\Phi_\frakF^\beta$ is an isomorphism (epimorphism) for a direct family of monomorphisms $\frakF$ if and only if 
$\Phi_\frakF^{\overline{\beta}}$ and $\Phi_\frakF^{\pi_K}$ 
 are isomorphisms (epimorphisms).
\end{theorem}

\begin{proof}
In order to prove the equivalence, let us remark, using the fact that direct limits are exact in $\CC$, that
for every direct family $\frakF$ we have a commutative diagram
\begin{equation*} \xymatrix{          
0\ar[r] & \underrightarrow{\lim}\Dev_{\overline{\beta}}(M_i) \ar[d]^{\Phi_{\overline{\beta}}^\frakF} \ar[r] & 
\underrightarrow{\lim}\Dev_{\beta}(M_i) \ar[d]^{\Phi_{\beta}^\frakF} \ar[r] & 
\underrightarrow{\lim}\Dev_{\pi_K}(M_i) \ar[d]^{\Phi_{\pi_K}^\frakF} \ar[r]&
 0 \\ 
0\ar[r] & \Dev_{\overline{\beta}}(\underrightarrow{\lim}M_i)  \ar[r] & 
\Dev_{\beta}(\underrightarrow{\lim}M_i)  \ar[r] & 
\Dev_{\pi_K}(\underrightarrow{\lim}M_i)  \ar[r]&
 0 ,\\
	}
	\end{equation*}
and $\Phi_{\pi_K}^\frakF$ is monic by Proposition \ref{monomorphism-unions}. 
Now the equivalence stated in this corollary is obvious.
\end{proof}

The condition $\Phi^{\pi_K}_\frakF$ is an epimorphism can be replaced by a factorization condition: 

\begin{lemma}\label{epi-dev-2}
Let $\beta:L\to P$ be a homomorphism and $\frakF=(M_{i}, v_{ij})_{i,j\in I}$ a direct system.
Consider the following statements:
\begin{enumerate}
 \item $\Phi^\beta_\frakF$ is an epimorphism;
 \item \begin{enumerate}
	\item if $\iota_K:K\to L$ is the kernel of $\beta$, then for every homomorphism 
	$f:L\to \underrightarrow{\lim}M_i$ there exists $i\in I$ and $h:L\to M_i$ such that $f\iota_K=v_i h\iota_K$ 
	(i.e. the restriction of $f$ to factorizes through the canonical map $v_i$),
	\item if $\overline{\beta}:L/K\to P$ is induced by $\beta$ then $\Phi^{\overline{\beta}}_\frakF$ is an epimorphism.
	\end{enumerate}
\end{enumerate}
Then $(2)\Rightarrow (1)$. If all $v_i$ are monomorphisms we have $(1)\Leftrightarrow (2)$.
\end{lemma}

\begin{proof}
$(2)\Rightarrow (1)$. Let $f:L\to \underrightarrow{\lim}M_i$ be a homomorphism. Using (a) we can find 
$i\in I$ and $h:L\to M_i$ such that $f\iota_K=v_i h\iota_K$. Then $(f-v_ih)\iota_K=0$, hence there exists 
$\overline{\delta}:L/K\to \underrightarrow{\lim}M_i$ such that $\overline{\delta}\pi=f-v_i h$. 

Using (b) and Lemma \ref{epi-dev}, 
we can find $j\in I$, $g:P\to \underrightarrow{\lim}M_i$ and $\gamma: L/K\to M_j$ such that 
$\overline{\delta}=g\overline{\beta}+v_j\gamma$. We can suppose $i=j$. Then 
$f-v_i h=g\overline{\beta}\pi+v_i\gamma \pi$, hence $f=g\beta+v_i(\gamma\pi+h)$. 
Another application of Lemma \ref{epi-dev} completes the proof.


$(1)\Rightarrow (2)$ Let $f:L\to \underrightarrow{\lim}M_i$ be a homomorphism. 
Using Lemma \ref{epi-dev} there exist $h:L\to M_i$ and $g:P\to \underrightarrow{\lim}M_i$ are homomorphisms such that 
$f=g \beta+v_i h$. Then 
$f\iota_K=v_i h\iota_K$, hence (a) is valid.

The condition (b) follows from Corollary \ref{cor-theorem-dev-restriction}.
\end{proof}

\begin{remark}
In fact the condition (b) in the above lemma can be proved directly.
In order to do this, let us consider a homomorphism $\overline{f}:L/K\to \underrightarrow{\lim}M_i$. If $\pi:L\to L/K$ 
is the canonical projection then we can find $i\in I$ and two homomorphisms $h:L\to M_i$, $g:P\to \underrightarrow{\lim}M_i$ 
such that $f=g \beta+v_i h$. Then $v_i h(K)=0$. Since $v_i$ is a monomorphism, we have $h(K)=0$. 
It follows that there exists $\overline{h}:L/K\to M_i$ 
such that $h=\overline{h}\pi.$ 
Then $(g\overline{\beta}+v_i\overline{h})\pi=\overline{f}\pi$, hence $g\overline{\beta}+v_i\overline{h}=\overline{f}$, 
and the proof is complete.
\end{remark}

{
In fact the case when $\Phi_\frakF^\beta$ is an epimorphism for all direct systems of monomorphisms can be characterized 
in the following way:

\begin{theorem}\label{epi-kernel}
Let $\beta:L\to P$ be a homomorphism in $\CC$. The following are equivalent:  
\begin{enumerate}
 \item for every direct system $\frakF=(M_{i}, v_{ij})_{i,j\in I}$ of monomorphisms $\Phi^\beta_\frakF$ is an epimorphism;
 \item \begin{enumerate}
	\item if $\iota_K:K\to L$ is the kernel of $\beta$, then $K$ can be embedded in a finitely generated 
	subobject $H\leq L$, and
	\item if $\overline{\beta}:L/K\to P$ is induced by $\beta$ then $\Phi^{\overline{\beta}}_\frakF$ is an epimorphism
	for all direct systems of monomorphisms $\frakF=(M_{i}, v_{ij})_{i,j\in I}$.
	\end{enumerate}
\end{enumerate}
\end{theorem}

\begin{proof}
(1)$\Rightarrow$(2) We apply Lemma \ref{epi-dev-2} to obtain (b). For (a), we apply Theorem \ref{unions-1} to find
a homomorphism $g:P\to L$ such that $\Ima(1_L-g\beta)$ can be embedded in a finitely generated subobject $H$ of $L$. 
Then $K$ can be also embedded in $H$.

(2)$\Rightarrow$(1) It is enough to prove that for every direct system of monomorphisms and for every
$f:L\to \underrightarrow{\lim}M_i$ there exists $i\in I$ and $h:L\to M_i$ such that $f\iota_K=v_i h\iota_K$. 

Let $f:L\to \underrightarrow{\lim}M_i$ be a homomorphism. 
By (a) there exists a factorization $\iota_K=\iota_H\iota_{KH}$. Since $H$ is finitely generated there exists an index 
$i\in I$ such that $i_Hf$ factorized through $v_i$. Therefore there exists $h:L\to M_i$ such that $f\iota_H=v_i h$,
hence $f\iota_k=v_i h\iota_K$.
\end{proof}
}



In the end of this section we come back to the general case, in order to characterize the functor $\Dev_\beta$
associated to a monomorphism $\beta:L\to P$ for the 
case when we can find a subobject $H\leq L$ such that the induced homomorphism $\overline{\beta}:L/H\to P/\beta(H)$ 
is split mono. 

\begin{proposition}
Let $\beta:L\to P$ be a monomorphism and $H$ a subobject of $L$. If $\overline{\beta}:L/H\to P/\beta(H)$ is the 
homomorphism induced by $\beta$ then 
we have an exact sequence of functors
$$(P/\beta(H),-)\overset{\overline{\beta}^*}\to (L/H,-)\to \Dev_\beta\to \Dev_{\iota_{\beta(H)}}\to \Dev_{\iota_H}\to 0,$$
and the following are equivalent:
\begin{enumerate}
 \item the induced 
homomorphism $\overline{\beta}:L/H\to P/\beta(H)$ is splitting monomorphism;

\item the induced  
sequence of functors $$0\to \Dev_\beta\to \Dev_{\iota_{\beta(H)}}\to \Dev_{\iota_H}\to 0$$
is exact.
\end{enumerate} 
\end{proposition}

\begin{proof}
%
%
We have a commutative diagram
\[\xymatrix{          
& 0\ar[d] & 0\ar[d]&  & \\
0\ar[r]& H\ar[d]^{\iota_H} \ar[r]^{\beta'} & \beta(H)\ar[d]^{\iota_{\beta(H)}}\ar[r]& 0 &  \\
 0\ar[r] & L \ar[r]^\beta \ar[d]^{\pi_H} & P\ar[r] \ar[d]^{\pi_{\beta(H)}} & M \ar@{=}[d] \ar[r] & 0\\
  0\ar[r] & L/H \ar[r]^{\overline{\beta}} \ar[d]	& P/\beta(H)\ar[r] \ar[d] & M \ar[r] & 0\\ 
	  & 0 & 0&  & \\
	}\]
with exact sequences, which induces a the solid part of the following commutative diagram of functors and natural transformations
\[\xymatrix{          
&  & 0\ar[d]& 0\ar[d] &  &\\
0\ar[r] & (M,-) \ar@{=}[d] \ar[r] & (P/\beta(H),-) \ar[d]^{\pi_{\beta(H)}^*} \ar[r]^{\overline{\beta}^*} & 
(L/H,-) \ar@{-->}[r] \ar[d]^{\pi_H^*}	&  \Dev_\beta(-)\ar@{=}[d] &\\ 
0\ar[r] & (M,-)  \ar[r] & (P,-) \ar[d]^{\iota_{\beta(H)}^*} \ar[r]^{\beta^*} & (L,-) \ar[r]^{\xi_\beta} \ar[d]^{\iota_H^*}	
& \Dev_\beta(-) \ar[r] \ar[d]^0 & 0\\ 
      & 0\ar[r] & (\beta(H),-)\ar[r]^{\beta'^*}\ar[d]^{\xi_{u_{\beta(H)}}}&(H,-) \ar[d]^{\xi_{u_H}} \ar[r] & 0\ar@{=}[d] \\
		&  & \Dev_{\iota_{\beta(H)}}(-)\ar[r]\ar[d]& \Dev_{\iota_H}(-)\ar[d]\ar[r] & 0  \\
		&  & 0& 0 & \\
	}\]
 in which all lines and columns are exact sequences. 
 Applying the snake lemma we obtain the natural transformation $(L/H,-)\dashrightarrow \Dev_\beta(-)$ such that 
 the sequence 
 $$(P/\beta(H),-)\overset{\overline{\beta}^*}\to (L/H,-)\dashrightarrow \Dev_\beta\to 
 \Dev_{\iota_{\beta(H)}}\to \Dev_{\iota_H}\to 0$$
 is exact.
 
 Now the equivalence (1)$\Leftrightarrow$(2) is obvious since $\overline{\beta}$ is split mono iff the natural 
 homomorphisms $\Hom_\CC(\overline{\beta},X)$ are epimorphisms for all $X\in\CC$. 
 
\end{proof}

\section{Commuting with direct sums} \label{sect-dev-direct-sums}

Let $\frakF=(M_{i})_{i\in I}$ is a family of objects in $\CC$ and $\nu_i:M_{i}\to \bigoplus_i M_i$ are the canonical monomorphisms.
Recall that 
\[
\Phi^{\beta}_\frakF:\bigoplus_i\Dev_\beta(M_i)\to \Dev_\beta(\bigoplus_i M_i)
\]
denotes the natural homomorphisms induced by the family
$\Dev_\beta(\nu_i)$, $i\in I$. It is easy to see that $\Phi^{M}_\frakF$ is a monomorphism,
since $(\Psi^{M}_\frakF)^{-1}(\Ima(\Hom(\beta,\bigoplus_i M_i)= \bigoplus_i \Ima(\Hom(\beta,M_i))$.
Moreover, $\Dev_\beta$ commutes with respect to finite direct sums (it is additive).

For every family of objects $(M_i,i\in I)$ and for $J\subset I$ denote by $\pi_{J}$ the canonical projection $\bigoplus_{i\in I} M_i\to \bigoplus_{i\in J} M_i$.

\begin{theorem}\label{oplus}
If $\beta:L\to P$ is a homomorphism  and $(M_i,i\in I)$ a family of objects, the following are equivalent: 
\begin{enumerate}
 \item $\Dev_\beta$ commutes with respect to direct sum of $(M_i,i\in I)$.
 \item For every homomorphism 	$f:L\to \bigoplus_{i\in I} M_i$ there exists a finite subset $F\subset I$, and $g:P\to \bigoplus_{i\in I\setminus F} M_i$ 
	such that $\pi_{I\setminus F}f=g \beta$. 
\end{enumerate}
\end{theorem}

\begin{proof} Since $\bigoplus_{i\in I}M_i$ is an inverse limit of the system $(M_F, F\in I^{<\omega})$ with canonical inclusions,
$\Phi^\beta_\frakF$ is a monomorphism. Now it remains to apply Lemma~\ref{epi-dev}.
\end{proof}

If $I$ is a set, $X\subseteq I$, and $M_i$, $i\in I$, is a family of objects,
we denote by $$\Pi_X^I:\Dev_\beta(\bigoplus_{i\in I}M_i)\to \Dev_\beta(\bigoplus_{i\in X}M_i)$$ the canonical epimorphism 
which is induced by the canonical map $\bigoplus_{i\in I}M_i\to \bigoplus_{i\in X}M_i$. Note that $\Pi_X^I$ 
is a splitting epimorphism of abelian groups.

Using a standard set-theoretical argument under assumption $(V=L)$ we prove that commuting of the functor 
$\Dev_\beta$ with countable direct sums
is equivalent to commuting with arbitrary direct sums.
First, we make an easy observation:

\begin{lemma}\label{L0} Let $M_i$, $i\in I$, be a family of objects. 
Then $\Dev_\beta$ commutes with $\bigoplus_{i\in I}M_i$ if and only if
for every $\epsilon\in \Dev_\beta(\bigoplus_{i\in I}M_i)$ there is a finite subset 
$F\subseteq I$ such that $\Pi_{I\setminus F}^I(\epsilon)=0$. 
\end{lemma}

{A cardinal $\lambda=|I|$ is \textsl{$\omega$-measurable} 
if it is uncountable and there exists a 
countably-additive, non-trivial, $\{0,1\}$-valued measure $\mu$ on the power set of $I$ such that $\mu(I)=1$ 
and $\mu(\{x\})=0$ for all $x\in I$. We recall that if such a cardinal exists then there exists a smallest $\omega$-measurable
cardinal $\mu$ and all cardinals $\lambda\geq \mu$ are also $\omega$-measurable.
}

\begin{proposition}\label{measurable}
Let $\kappa$ be a cardinal less than the first $\omega$-measurable cardinal. If $\Dev_\beta$ commutes with respect to countable direct sums then   
$\Dev_\beta$ commutes with respect to direct sums of $\kappa$ objects.
\end{proposition}

\begin{proof}
Let $K_i$, $i\in I$, be a family of modules such that $I$ is of cardinality $\kappa$  and  $\epsilon\in \Dev_\beta(\bigoplus_{i\in I}K_i)$ is a fixed extension.
By Lemma~\ref{L0} it is enough to prove that there is a finite subset $F\subseteq I$ such that 
$\Pi_{I\setminus F}^I(\epsilon)=0$. 
Let consider the set 
$$\mathcal{I}(I)=\{X\subseteq I \mid \textrm{there is a finite subset } F\subseteq X\ \textrm{such that } 
\Pi^I_{X\setminus F}(\epsilon)=0\}.$$

Suppose that $I\notin \mathcal{I}(I)$. 
We claim that there exists a subset $Y\subseteq I$ such that $\mathcal{I}(Y)$ is a non-principal $\omega_1$-complete ideal. 
Let us observe that for every subsets 
$F\subseteq X\subseteq I$ we have $\Pi^I_{X\setminus F}=\Pi^X_{X\setminus F}\Pi^I_{X}$. 
Furthermore, it is not hard to see that 
$\mathcal{I}(I)$ contains $\emptyset$ and it is closed with respect to subsets and finite unions. In order to complete the proof 
of our claim it is remains to prove that if $X_n$, $n\in\N$, is a countable set of pairwise disjoint subsets of $I$ then there
exists $n_0\in \N$ such that $\bigcup_{n\geq n_0}X_n\in \mathcal{I}(I)$. 

Let $X_n$, $n\in\N$, be a family of pairwise disjoint subsets of $I$, and $X=\bigcup_{n\in\N}X_n$. Since $\Dev_\beta$ 
commutes with respect to countable direct sums, the canonical homomorphism 
$$\bigoplus_{n\in \N}\Dev_\beta(\bigoplus_{i\in X_n}K_i)\to \Dev_\beta(\bigoplus_{n\in \N}(\bigoplus_{i\in X_n}K_i))=
\Dev_\beta(\bigoplus_{i\in X}K_i)$$
is an isomorphism. Therefore there is a positive integer $n_0$ such that 
$$\Pi^I_{\bigcup_{n\geq n_0}X_n}(\epsilon)=\Pi^X_{\bigcup_{n\geq n_0}X_n}(\Pi^I_X(\epsilon))=0,$$ 
 hence $\bigcup_{n\geq n_0}X_n\in\mathcal{I}(I)$. 


{
Now we claim that there exists $Y\notin\CI(I)$ such that for 
every subset $Z\subset Y$ with $Z\notin \mathcal{I}(I)$ we have  $Y\setminus Z\in \mathcal{I}(I)$.

Suppose by contradiction that such a $Y$ does not exists. It follows that for every $Y\notin \CI(I)$ there
exists a nonempty subset $Z\subseteq Y$ such that $Z,Y\setminus Z\notin \CI(I)$. Since 
$\emptyset\neq I\notin \CI(I)$ we can find a partition $I=Z_0\cup Y_1$ such that $Z_0,Y_1\notin \CI(I)$. 
Now $\emptyset\neq Y_1\notin\CI(I)$, hence there exists a partition $Y_1=Z_1\cup Y_2$ such that $Z_1,Y_2\notin \CI(I)$.
We continue inductively this kind of choice: if $Y_n$ is constructed then  
exists a partition $Y_n=Z_n\cup Y_{n+1}$ such that $Z_n,Y_{n+1}\notin \CI(I)$. Therefore we obtain a countable sequence 
of sets $Z_n\notin \CI(I)$, and it is not hard to see that these sets are pairwise disjoint. But, by what we proved so far,
there exists $n_0$ such that $\bigcup_{n\geq n_0}Z_n\in\mathcal{I}(I)$, a contradiction.

Then there exists a subset $y\subseteq I$ such that $Y\notin\CI(I)$ and for 
every subset $Z\subset Y$ with $Z\notin \mathcal{I}(I)$ we have  $Y\setminus Z\in \mathcal{I}(I)$.
It is not hard to see that we can define an $\omega$-additive 
$\{0,1\}$-valuated map $\mu$ on the power-set of $I$ via the rule $\mu(U)=1$ if $U\cap Y\notin \mathcal{I}(I)$, 
and $\mu(U)=0$ otherwise. It follows that $I$ is $\omega$-measurable, a contradiction.
}\end{proof}

\begin{corollary}
Assume $(V=L)$. If $\Dev_\beta(-)$ commutes with respect to countable direct sums then   
$\Dev_\beta(-)$ commutes with respect to all direct sums.
\end{corollary}

It is well-known that $\Hom(M,-)$ commutes with respect to countable direct sums iff it commutes with respect to all direct sums.
Furthermore, as a consequence of the previous result and Example~\ref{dev} we obtain 

\begin{corollary}
Let $(V=L)$ and $M\in \CC$. Then $\Ext^1_\CC(M,-)$ commutes with respect to countable direct sums if and only if
$\Ext^1_\CC(M,-)$ commutes with respect to all direct sums.
\end{corollary}

\begin{remark}
We don't know what is happen if we remove the set theoretic assumption $(V=L)$. For the case of abelian groups it can be 
proved, as in \cite[Theorem 5.3]{ABS} that if $M$ is an abelian group such that 
$\Ext^1_{Ab}(M,-)$ commutes with respect to countable direct sums then $M$ is an Whitedead group. On the other side,
the same result show us that if $\Ext^1_{Ab}(M,-)$ commutes with respect to all direct sums then $M$ is free. The 
interested reader can find some similar phenomena in \cite[Section 2]{ABS1}.
\end{remark}

The following results characterizes when $\Dev_\beta$ commutes with respect to countable direct sums. 
{
It generalizes a classical characterization of small modules proved by Rentschler in \cite{Re}.
}

\begin{proposition}\label{def-omega} Let $\beta:L\to P$ be a homomorphism. 
The functor $\Dev_\beta$ commutes with respect to countable direct sums if and only if for every countable chain of
subobjects $$(D): L_0\hookrightarrow L_1\hookrightarrow L_2\hookrightarrow\dots$$ such that $L$ is a direct union of $(D)$
there exists $n$ for which the induced map $\beta': L/L_n\to P/\beta(L_n)$ is a splitting monomorphism.
\end{proposition}

\begin{proof} ($\Rightarrow$) For each $i\in \N$ denote by $\iota_i:L_i\to L$ the canonical monomorphism 
and put $A_i=\Coker(\iota_i)\cong L/L_i$.
Suppose that  $\sigma:L\to \bigoplus_iA_i$ is the morphism defined by direct sum of the canonical epimorphisms 
$\rho_i:L\to A_i$, i.e. $\pi_i\sigma=\rho_i$, where $\pi_i:\bigoplus_iA_i\to A_i$ is the canonical projection. 
By the hypothesis, there exists a finite subset $F\subseteq \mathbb{N}$ and a homomorphism 
$g:P\to \bigoplus_{i\in \N\setminus F}A_i$ 
such that $\pi_{\N\setminus F}\sigma=g\beta$.
Let $n\notin F$. If $\rho_{n}:L\to A_n$ represents the canonical epimorphism, we obtain $\rho_n=\pi_ng\beta$, hence 
$\pi_n g \beta(L_n)=0$. Then $\pi_ng$ factorizes through the canonical epimorphism $\mu_n:P\to P/\beta(L_n)$, so 
$\pi_ng= g'\mu_n$ and $g':P/\beta(L_n)\to A_n$. 

We obtain $g'\beta'\rho_n=g'\mu_n\beta=\pi_n g\beta=\rho_n$, hence $g'\beta'=1_{A_n}$, and the proof is complete.

($\Leftarrow$) Let $\sigma\in\Hom(L,\bigoplus_{i<\omega} A_i)$ and denote by $\pi_{\ge n}$ the canonical epimorphisms 
$\bigoplus_{i} A_i\to \bigoplus_{i\ge n} A_i$.  Obviously, the family $L_n=\Coker(\pi_{\ge n}\alpha)$ with canonical monomorphisms forms an increasing chain 
such that $L$ is its direct union.

By the hypothesis there exists $n$ such that $\beta':L/L_n\to P/\beta(L_n)$ has a left inverse $g':P/\beta(L_n)\to L/L_n$. 

If we put $F=\{1,\dots,n\}$ and $g:P\to \bigoplus_{i>n}A_i$, $g=g'\mu_n$, where $\mu_n:P\to P/\beta(L_n)$ is the 
canonical epimorphism, then we can apply Theorem \ref{oplus} to obtain the conclusion.
\end{proof}


We will
say that the homomorphism $\beta:L\to P$ is {\it $\kappa$-splitting small}, where $\kappa$ is a cardinal, 
if for every system of objects 
$(A_i, i<\kappa)$ and for every homomorphism $\sigma:L\to \bigoplus_{i<\kappa}A_i$ there exists a finite subset $F\subset\kappa$ such that
the cokernel homomorphism $\rho$ in the pushout diagram
\begin{equation*}
\begin{CD}
L@>{\beta}>>P@>>> U @>>> 0 \\
@V{\pi_{\kappa\setminus F}\sigma}VV@ VVV @|@.\\
\bigoplus_{i\in \kappa\setminus F}A_i@>>> X @>{\rho}>> U @>>> 0
\end{CD} \tag{D2}
\label{D3}
\end{equation*}
splits.

Now, we make an elementary observation.

\begin{lemma}\label{splitting} Consider a commutative diagram with exact rows and columns
\begin{equation}
\begin{CD}
A_0@>{\nu_0}>> B_0@>{\pi_0}>> C@>>>0 \\
@V{\alpha_0}VV @V{\beta_0}VV @|@.@.\\
A_1@>{\nu_1}>> B_1@>{\pi_1}>> C@>>>0
\end{CD} \tag{D3} \label{D2}.
\end{equation}
If  $\overline{\pi_0}:B_0/\Ker \beta_0\to C\to 0$ induced by \eqref{D2} splits,
then $\pi_1$ splits as well.
\end{lemma}
\begin{proof} Let $\overline{\beta}_0:B_0/\Ker \beta_0\to B_1$ the homomorphism induced by $\beta_0$.
Since there exists $\rho:C\to B_0/\Ker \beta_0$ such that $\id_C=\overline{\pi_0}\rho=\pi_1\overline{\beta_0}\rho$, the homomorphism $\pi_1$ splits.
\end{proof}

\begin{proposition}\label{omega} The homomorphism $\beta$ is  $\omega$-splitting small if and only if for 
$\Dev_\beta$ commutes with respect to countable direct sums.
\end{proposition}

\begin{proof} ($\Rightarrow$) 
We consider a countable chain of
subobjects $$(D): L_0\hookrightarrow L_1\hookrightarrow L_2\hookrightarrow\dots$$ such that $L$ is a direct union of $(D)$.

For each $i$ denote by $\iota_i:L_i\to L$ the canonical monomorphism and put $A_i=\Coker(\iota_i)\cong L/L_i$.
Suppose that  $\sigma:L\to \bigoplus_iA_i$ is the morphism defined by direct sum of the 
canonical projections.
By the hypothesis, there exists a finite subset $F$ such that the homomorphism $\rho$ in the pushout diagram $\eqref{D3}$ splits. 
Let $n\notin F$. Clearly, the homomorphism $\overline{\rho}$ in the pushout diagram
\begin{equation}
\begin{CD}
\bigoplus_{i\in \kappa\setminus F}A_i@>>> X@>{\rho}>> U @>>> 0 \\
@VVV @VVV@|@.\\
A_n@>>> Y@>{\overline{\rho}}>> U @>>> 0
\end{CD}
\tag{D4} \label{D4}
\end{equation}
splits. Since $\pi_n$ is an epimorphism and the composition of the pushout diagrams $\eqref{D3}$ and $\eqref{D4}$ is so, the diagram 
\[
\begin{CD}
L@>{\beta}>>P@>>> U @>>> 0 \\
@V{\pi_n}VV @V{\tau}VV@|@.\\
A_n@>>> Y@>{\overline{\rho}}>> U @>>> 0
\end{CD}
\]
commutes and $\tau$ is an epimorphism. As $\overline{\rho}$ splits, it remains to observe that 
{$\Ker\tau=\beta(L_n)$.}

($\Leftarrow$) Let $\sigma\in\Hom(L,\bigoplus_{i<\omega} A_i)$ and denote by $\pi_{\ge n}$ the canonical epimorphisms 
$\bigoplus_{i} A_i\to \bigoplus_{i\ge n} A_i$.  Obviously, 
the family $L_n=\Coker(\pi_{\ge n}\alpha)$ with canonical monomorphisms forms an increasing chain 
such that $L$ is its direct union.

By the hypothesis there exists $n$ such that $\beta(L)/\beta(L_n)$ is a direct summand of $P/\beta(L_n)$, hence
it remains to put $F=\{1,\dots,n\}$ and to apply Lemma~\ref{splitting} on the pushout diagram \eqref{D3}.
\end{proof}

Now we will apply the above results in order to see when $\Ext_\CC^1(M,-)$
commutes with respect to direct sums.

\begin{lemma}\label{L2} Let $\kappa$ be a cardinal, and consider an exact sequence $L \overset{\beta}\to P\to M\to 0$.
Then $\beta$ is  $\kappa$-splitting small if and only if $\Dev_\beta$ commutes with respect to direct sums of $\kappa$ objects.
\end{lemma}

\begin{proof} ($\Rightarrow$) Let $A_i,\ i\leq \kappa$ be a family of objects.
 If $\sigma\in\Hom_\CC(L,\bigoplus_{i<\kappa} A_i)$ then we consider the pushout diagram
\begin{equation}
\begin{CD}
L@>{\beta}>> P@>{\alpha}>> M@>>>0 \\
@V{\sigma}VV @V{\rho}VV @|@.@.\\
\bigoplus_{i<\kappa} A_i@>{\nu}>> X @>{\pi}>> M@>>>0.
\end{CD}\tag{D5}\label{D5}
\end{equation}
As $\beta$ is  $\kappa$-splitting small, there exists a finite set $F\subset \kappa$ such that the second row of the pushout diagram
\begin{equation}
\begin{CD}
L@>{\beta}>> P@>{\alpha}>> M@>>>0 \\
@V{\pi_{\kappa\setminus F}\sigma}VV @VVV @|@.@.\\
\bigoplus_{i\in\kappa\setminus F} A_i@>>> Y @>>> M@>>>0
\end{CD}\tag{D6}\label{D7}\ \ .
\end{equation}
splits. Thus $\Pi^\kappa_{\kappa\setminus F}(E)=0$. Now the assertion follows from Lemma~\ref{L0}.

($\Leftarrow$) Fix $\sigma:L\to \bigoplus_{i<\kappa}A_i$. Then there exist an object $X$ and homomorphisms $\rho$ and $\nu$ such that \eqref{D5} forms 
a pushout diagram. Since $\Dev_\beta$ commutes with respect to the direct sums of family $(A_i, i<\kappa)$, 
 Lemma \ref{L0} imply that there exists a finite subset $F\subset\kappa$ such that
the second row of the pushout diagram 
\[
\begin{CD}
0@>>>\bigoplus_{i<\kappa} A_i@>{\nu}>> X @>{\pi}>> M@>>>0 \\
@.@V{\pi_{\kappa\setminus F}}VV @VVV @|@.@.\\
0@>>>\bigoplus_{i\in\kappa\setminus F} A_i@>>> Y @>>> M@>>>0
\end{CD}\label{D6}
\]
splits. Thus we get pushout diagram \eqref{D7} where the second row splits, so $\beta$ is  $\kappa$-splitting small.
\end{proof}

\begin{theorem}\label{thm2}  Let $\kappa$ be a cardinal less than the first $\omega$-measurable cardinal. The 
following are equivalent for a homomorphism $0\to L \overset{\beta}\to P$ in $\CC$:
\begin{enumerate}
\item The functor $\Dev_\beta$ commutes with respect to direct sums of $\kappa$ objects,
\item $\beta$ is  $\kappa$-splitting small,
\item $\beta$ is  $\omega$-splitting small.
\end{enumerate}
\end{theorem}

\begin{proof} $(1)\Leftrightarrow(2)$ is the assertion of Lemma~\ref{L2}, $(2)\Rightarrow(3)$ is trivial and $(3)\Rightarrow(1)$ follows from Proposition~\ref{measurable}
and Lemma~\ref{L2}.
\end{proof}

We can apply the last assertion to see when an $\Ext^1$-covariant functor preserves direct sums.

\begin{corollary}\label{thm-ext-sums} 
Suppose that $\CC$ has projective strong generator which is a direct sum of finitely 
presented objects. Let $\kappa$ be a cardinal less than the first $\omega$-measurable cardinal. The 
following are equivalent for a projective presentation $0\to L \overset{\beta}\to P\to M\to 0$ of $M\in\CC$:
\begin{enumerate}
\item The functor $\Ext^1_\CC(M,-)$ commutes with respect to direct sums of $\kappa$ objects,
\item $\beta$ is  $\kappa$-splitting small,
\item $\beta$ is  $\omega$-splitting small.
\end{enumerate}
\end{corollary}

\begin{corollary}
Let $(V=L)$, and suppose  that $\CC$ has projective strong generator which is a direct sum of finitely 
presented objects. If $M\in\CC$ then $\Ext^1(M,-)$ commutes with respect to all direct sums  if and only if 
there exists a projective presentation $0\to L \overset{\beta}\to P\to M\to 0$ for $M$ such that $\beta$ is  
$\omega$-splitting small.

In these conditions for all projective presentations $0\to L' \overset{\beta'}\to P'\to M\to 0$ and for 
all cardinals $\kappa$ the homomorphism
$\beta'$ is $\kappa$-splitting small.
\end{corollary}

We close this section with an application of Proposition \ref{basic-def} to the study of the covariant 
$\Ext^1$-functor.

\begin{lemma}\label{small}
Let $0\to L \to P\to M\to 0$ be an exact sequence such that $P$ is projective. 
If $L$ and $M$ are small objects, then $M$ and $P$ are finitely generated and $\Ext^1(M,-)$ commutes with direct sums.
\end{lemma}

\begin{proof}
Since the class of small objects is closed with respect to extensions, $P$ is small, hence finitely generated.
Note that a direct sum $\bigoplus_{i<\kappa}A_i$ is precisely direct union of the diagram $\frakF=(\bigoplus_{i\in F}A_i, \nu_{FG}|\ F\subseteq G\in \kappa^{<\omega})$
where $\nu_{FG}$ are the canonical inclusions $\bigoplus_{i\in F}A_i\to \bigoplus_{i\in G}A_i$. 
As all homomorphisms $\Psi_\frakF$ from the diagram \eqref{D10} are isomorphisms, $\Phi_\frakF^M$ is isomorphism as well.
\end{proof}

Applying Proposition \ref{basic-def}(1) we obtain the following

\begin{corollary} 
Suppose  that $\CC$ has projective strong generator which is a direct sum of finitely 
presented objects.
Let $M$ be a finitely generated object, and let $0\to L \overset{\beta}\to P\to M\to 0$ be 
a projective presentation for $M$ such that $P$ is finitely generated. Then $\Ext^1_\CC(M,-)$ commutes with respect to direct sums if and only
if for  $L$ is small.
\end{corollary}

\section{The covariant $\Ext^1$-functor and direct unions}\label{Ext-unions}

\textsl{In this section $\CC$ will be an abelian category with a projective strong generator which is a direct sum of finitely 
presented objects.} 

Let us fix an object $M$ in $\CC$. We will apply the previous results to 
study the commuting properties for the covariant functor 
$\Ext_\CC^1(M,-)$. In order to do this, we fix a projective presentation $$0\to L\overset{\beta}\to P\to M\to 0,$$ 
and we can apply the
previous results for the functor $\Dev_\beta$.

In order to simplify our presentation we will say, as in
\cite{Br-ext-lim}, that the object $M$ is an
\textsl{$\mathrm{fg}$-$\Omega^1$-object} (respectively
\textsl{$\mathrm{fp}$-$\Omega^1$-object}) if there is a projective
resolution
$$(\boldP):\ \dots\to P_2\to P_1\overset{\alpha_1}\to P_0\to M\to 0$$ such that
$\Omega^1(\boldP)=\Ima(\alpha_1)$ is finitely generated
(respectively, finitely presented), i.e. there is a projective
resolution for $M$ such that the first syzygy associated to this
resolution is finitely generated (finitely presented). The object $M$ is an \textsl{$FP_n$-object}
if it has a projective resolution such that $P_i$ are finitely presented for all $i\in\{0,\dots,n\}$.

It is proved in \cite[Lemma
3.1.6]{GT06} that if $M$ is an $FP_2$ object then $\Ext^1_\CC(M,-)$ commutes with respect to direct limits.
Using Proposition \ref{basic-def} and Example~\ref{dev}(1) for the kernel of a projective presentation $P\to M\to 0$,
it is easy to see that for the finitely generated objects this hypothesis is sharp. 
{We recall that in our hypothesis every finitely generated projective object is finitely presented.}

\begin{corollary}\label{basic}
Let $M$ be a finitely generated object. \begin{enumerate} \item
$\Ext^1_\CC(M,-)$ commutes with respect to direct unions if and only
if $M$ is finitely presented.

\item $\Ext^1_\CC(M,-)$ commutes with respect to direct limits if and
only if $M$ is an $FP_2$-object.
\end{enumerate}
\end{corollary}

Inductively, using the dimension shifting formula we obtain a
version, of \cite[Theorem 2]{Brown} and \cite[Theorem A]{Strebel}:

\begin{corollary}\label{cor-br-str}
The following are equivalent for an object $M$:
\begin{enumerate}
\item $M$ has a projective resolution $$(\boldP):\ P_n\to \dots\to
P_2\to P_1\overset{\alpha_1}\to P_0\to M\to 0$$ such that $P_i$
are finitely presented for all $0\leq i\leq n$;

\item The functors $\Ext_\CC^i(M,-)$ commute with respect to direct
unions for all $0\leq i\leq n$.
\end{enumerate}
\end{corollary}

\begin{remark}
Corollary \ref{basic} can be
reformulated in the following way: \textit{the functor
$\Hom_\CC(M,-)$ commutes with respect to direct limits if and only if
the functors $\Hom_\CC(M,-)$ and $\Ext_\CC^1(M,-)$ commute with
respect to direct unions}. The proof presented here uses the
existence of the strong generator $\mathcal{U}$ which is a direct sum of finitely 
presented projective objects. It is an open question if this result is
valid in more general settings, e.g. for general Grothendieck
categories without enough projectives. For instance this
equivalence is valid for the category of all Abelian $p$-groups
($p$ is a fixed prime), {which is a Grothendieck category without non-trivial projective objects,}
as a consequence of \cite[Theorem
5.4]{S11}. 
\end{remark}

In order to prove the main result of this section, 
we say that a covariant functor $F:\CC\to Ab$ is \textsl{isomorphic to a direct summand} of a functor 
$G:\CC\to Ab$ if we can find two natural transformations 
$\rho:F\to G$ and $\pi:G\to F$ such that $\pi\rho=1_F$. 

\begin{lemma}
Let $F,G:\CC\to Ab$ be additive covariant functors such that $F$ is isomorphic to a direct summand of $G$. If $\frakF=(M_i)_{i\in I}$ is a direct family such that the canonical 
homomorphism $\Phi_G: \underrightarrow{\lim}G(M_i)\to G(\underrightarrow{\lim}M_i)$ is monic (epic) then the canonical homomorphism $\Phi_F: \underrightarrow{\lim}F(M_i)\to F(\underrightarrow{\lim}M_i)$ is monic (epic). 
\end{lemma}

\begin{proof}
If $\rho:F\to G$ and $\pi:G\to F$ are natural transformations such that $\pi\rho=1_F$, we have the commutative diagram 
\[
\begin{CD}
0@>>>\underrightarrow{\lim} F(M_i)@>{\underrightarrow{\lim}\rho_{M_i}}>> \underrightarrow{\lim} G(M_i)@>{\underrightarrow{\lim}\pi_{M_i}}>> \underrightarrow{\lim} F(M_i)@>>>0 \\
@.@V{\Phi_F}VV @V{\Phi_G}VV @V{\Phi_F}VV@.@.\\
0@>>>F(\underrightarrow{\lim} M_i)@>{\rho_{\underrightarrow{\lim} M_i}}>> G(\underrightarrow{\lim} M_i) @>{\pi_{\underrightarrow{\lim} M_i}}>> F(\underrightarrow{\lim} M_i)@>>>0
\end{CD}
\ ,
\]
and the conclusion is now obvious.
\end{proof}

\begin{corollary}\label{sumanzi}
Let $M$ be an object such that $\Ext^1_\CC(M,-)$ commutes with
respect to a limit of a direct system $\frakF$. Then every direct summand $N$
of $M$ has the same property.
\end{corollary}

Now we are ready to characterize when the covariant $\Ext^1_\CC$ functors commute with respect to direct unions.
We recall that $M$ is \textsl{$2$-almost projective} if it is a direct summand of a direct sum $P\oplus F$ with $P$ 
a projective object and $F$ a finitely presented object, \cite{Drin}. For reader's convenience we include a proof for 
the following characterization.

\begin{theorem}\label{main-th}
The following are equivalent for an object $M$ in $\CC$:
\begin{enumerate}
\item The functor $\Ext^1_\CC(M,-)$ commutes with respect to direct
unions;

\item $M$ is a direct summand of an
$\mathrm{fg}$-$\Omega^1$-object.

\item $M$ is a 2-almost projective object.
%
\end{enumerate}
\end{theorem}

\begin{proof}
$(1)\Rightarrow(2)$ We consider a projective resolution $0\to L\overset{\beta}\to P\to M\to 0$. 
By Theorem \ref{unions-1} there exists a finitely generated object $H\leq L$ such that the induced 
homomorphism $\overline{\beta}:L/H\to P/\beta(H)$ is split mono. Since
$\Coker(\overline{\beta})\cong\Coker(\beta)\cong M$, it follows
that $M$ is isomorphic to a direct summand of the
$\mathrm{fg}$-$\Omega^1$-object $P/\beta(H)$.

(2)$\Rightarrow$(3) It is enough to assume that $M$ is an
$\mathrm{fg}$-$\Omega^1$-object. If $M$ is such an object then we
can consider the diagram \eqref{D10} with $P$ projective and $L$
finitely generated. If $U$ is an object such that $P\oplus U$ is a
direct sum of copies of some objects from $\CU$, we consider the
induced exact sequence $0\to L\overset{\beta}\to P\oplus
U\overset{\alpha\oplus 1_U}\longrightarrow M\oplus U\to 0$. Let
$P\oplus U=\oplus_{i\in I}P_i$, where all objects $P_i$ are
finitely presented and projective. Since $L$ is finitely
generated, there is a finite subset $J\subseteq I$ such that
$\beta(L)\subseteq \oplus_{i\in J}P_i$. Therefore $M\oplus U\cong
(\oplus_{i\in I\setminus J}P_i)\oplus (\oplus_{i\in
J}P_i)/\beta(L)$ is a direct sum of a projective object and a
finitely presented object.

(3)$\Rightarrow$(1) In view of Corollary \ref{sumanzi}, we can assume
that $M$ is finitely presented. Then we apply Corollary~\ref{basic}.
\end{proof}

For arbitrarily direct limits, we have not a general answer.
However, for some particular cases, including coherent categories
($\CC$ is \textsl{coherent} if every finitely generated subobject
of a projective object is finitely presented), we can apply the
previous result. In order to do this, let us state the following

\begin{proposition}
Let $M$ be an $\mathrm{fg}$-$\Omega^1$-object. The following are
equivalent:
\begin{enumerate}
\item $M$ is an $\mathrm{fp}$-$\Omega^1$-object; \item
$\Ext^1_\CC(M,-)$ commutes with respect to direct limits.
\end{enumerate}
\end{proposition}

\begin{proof}
(1)$\Rightarrow$(2) Suppose that $M$ is an $\mathrm{fp}$-$\Omega^1$-object.  As in the
proof for (2)$\Rightarrow$(3) in the previous theorem, we observe that there is a
projective object $L$ such that $M\oplus L$ is a direct sum of an
$FP_2$-object and a projective object. Therefore we can suppose
that $M$ is $FP_2$. For this case the result is well known (see
\cite[Lemma 3.1.6]{GT06}).

(2)$\Rightarrow$(1) let $0\to K\to P\to M\to 0$ be an exact sequence such
that $K$ is finitely generated. Using again the proof of
$(2)\Rightarrow(3)$ in the previous theorem, there is a projective
object $K$ such that $M\oplus K=N\oplus U$, $N=P'/K$, where $P'$
is a finitely generated projective object and $U$ is projective
such that $P'\oplus U=P\oplus K$. {Then $\Ext^1_\CC(N,-)\cong
\Ext^1_\CC(M,-)$ commutes with respect to direct limits. Since $P'$ is finitely presented 
we can use Lemma~\ref{basic-def}, and we conclude that $K$ is finitely presented.}
\end{proof}

From this proposition and its proof we obtain some useful
corollaries. First of them allows us to construct examples of
objects $M$ such that $\Ext_\CC^1(M,-)$ commutes with respect
to direct unions, but it does not commute with respect to direct limits.

\begin{corollary}\label{coherent-1}
Suppose that $M$ is an $\mathrm{fp}$-$\Omega^1$-object and $0\to
L\to P\to M\to 0$ is an exact sequence such that $P$ is finitely
generated projective. Then $L$ is finitely presented.

Consequently, if for every finitely presented object $M$ the
functor $\Ext^1_\CC(M,-)$ commutes with respect to direct limits then
$\CC$ is coherent.
\end{corollary}

\begin{proof}
{We consider an exact sequence $0\to
L_1\to P_1\to M\to 0$ such that $P_1$ is finitely presented projective and $L_1$ is finitely presented.
By Schanuel's lemma we have $L_1\oplus P\cong L\oplus P_1$, and now the conclusion is obvious.
}\end{proof}

In fact, for coherent categories (in particular for modules over
coherent rings or for the category of modules over the category
$\textrm{mod-}R$) the functor $\Ext^1_\CC(M,-)$ commutes with
respect to direct limits if and only it it commutes with respect
direct unions:

\begin{corollary}\label{coherent}
Suppose that $M$ has a projective resolution $(\boldP)$ such that
$\Omega^1(\boldP)$ is a direct union of finitely presented
subobjects. The following are equivalent:
\begin{enumerate}
\item $\Ext^1_\CC(M,-)$ commutes with respect to direct limits;
\item $\Ext^1_\CC(M,-)$ commutes with respect to direct unions;
\item $M$ is a direct summand of an
$\mathrm{fp}$-$\Omega^1$-object.
\end{enumerate}
\end{corollary}

\begin{proof}
In the proof of Theorem \ref{main-th} we can choose $\frakF$ such
that all $M_i$ are finitely presented.
\end{proof}

As each countably generated object is a direct union of a chain of finitely generated modules, Theorem~\ref{main-th}, Example~\ref{dev}(1) and the previous result implies the following consequence:

\begin{corollary}\label{du>ds} If $\Ext^1(M,-)$ commutes with direct sums and $L$ is countable generated, then $M$ is 2-almost projective,
hence $\Ext^1(M,-)$ commutes with direct unions.
\end{corollary}

Moreover, for coherent categories we obtain from Theorem \ref{main-th} a generalization of \cite[Theorem A]{Cor}:

\begin{theorem}\label{coherent-cor}
Suppose that $\CC$ is a coherent category. The following are equivalent for an object $M$ and a positive integer $n$:
\begin{enumerate}
 \item the $\Ext_\CC^n(M,-)$ commutes with respect to direct limits (unions);
 \item in $m\geq n$ is an integer then then $\Ext_\CC^m(M,-)$ commutes with respect to direct limits (unions).
\end{enumerate}
\end{theorem}

\begin{proof}
(1)$\Rightarrow$(2) By dimension shifting formula we can assume $n=1$, and it is enough to prove that $\Ext_\CC^2(M,-)$
commutes with respect to direct unions. 

Let $0\to L\overset{\beta}\to P\to M\to 0$ be an exact sequence such that $P$ is projective. By Theorem \ref{unions-1}(6) 
there exists a finitely generated subobject $H\leq L$ such that the induced homomorphism $\beta':L/H\to P/\beta(H)$ 
is split mono. Using Theorem \ref{main-th} we obtain that $\Ext_\CC^1(L/H,-)$ commutes with respect to direct limits. Moreover,
we can view $H$ as a finitely generated subobject of $P$, hence $H$ is finitely presented. Therefore for every 
directed family $\frakF=(M_{ij},v_{ij})$, in the commutative diagram  
\begin{equation*}
\xymatrix{
 \underrightarrow{\lim}(H,M_i) \ar[r] \ar[d]^{\Psi^H_\frakF} &
\underrightarrow{\lim}\Ext_\CC^1(L/H,M_i) \ar[r] \ar[d]^{\Phi^{L/H}_\frakF} &  \underrightarrow{\lim}\Ext_\CC^1(L,M_i)
\ar[d]^{\Phi^L_\frakF} \ar[r] & 
\underrightarrow{\lim}\Ext_\CC^1(H, M_i) \ar[d]^{\Phi^\beta_\frakF}  \\
 (H,\underrightarrow{\lim}M_i) \ar[r] &
\Ext_\CC^1(L/H,\underrightarrow{\lim}M_i) \ar[r] &
\Ext_\CC^1(L,\underrightarrow{\lim}M_i) \ar[r] & \Ext_\CC^1(H,\underrightarrow{\lim}M_i) }
\end{equation*}
the homomorphisms $\Psi^H_\frakF$, $\Phi^{L/H}_\frakF$ and $\Phi^{H}_\frakF$ are isomorphisms. Then $\Phi^{L}_\frakF$
is also an isomorphism, and the proof is complete.

(2)$\Rightarrow$(1) is obvious.
\end{proof}



\section{Ext-steady rings}

We say that the category $\CC$ is \textsl{finite ext-steady} if
for every finitely generated object $M$ such that $\Ext^1(M,-)$ commutes with all direct sums 
it holds that $\Ext^1(M,-)$ commutes with all direct unions.

\begin{proposition}\label{p1} The following conditions equivalent:
\begin{enumerate}
\item[(1)] $\CC$ is finite ext-steady,
\item[(2)] every small subobject of every  projective object is finitely generated,
\item[(3)] every small subobject of every finitely generated projective object is finitely generated.
\end{enumerate}
\end{proposition}

\begin{proof} (1)$\Rightarrow$(2) Let $L$ be a small subobject of a projective object $P$, i.e. there exist a cardinal $\kappa$, a family $P_i\in\CU$, $i<\kappa$, and a monomorphism
$L\to\bigoplus_{i<\kappa}P_i$. 
Moreover, as $L$ is small, there exists $F\subset\kappa$ and monomorphism $\iota:L\to\bigoplus_{i\in F}P_i$. 
Put $M=\Coker(\iota)$. By Lemma~\ref{small}, the functor $\Ext^1(M,-)$ commutes with direct sums, hence it commutes with respect to direct unions by the hypothesis. Thus $L$ is finitely generated
by Theorem~\ref{main-th}.

(2)$\Rightarrow$(3) Clear.

(3)$\Rightarrow$(1) Let $M$ be a finitely generated object such that $\Ext^1(M,-)$ commutes with direct sums.
Then homomorphisms $\Psi_\frakF^M$, $\Psi_\frakF^P$ and $\Phi_\frakF^M$  from the diagram \eqref{D10} are isomorphisms, $\Phi_\frakF^L$ is isomorphism as well, hence $L$ is small. 
By the hypothesis $L$ is finitely generated. Thus $\Ext^1(M,-)$ commutes with direct unions by Theorem~\ref{main-th}(4).
\end{proof}

We say that a unital ring $R$ is right finite ext-steady if the category of all right $R$-modules is finite ext-steady.

\begin{example}
It is proved in \cite{Tr} that every infinite product of unital rings contains an infinitely generated small ideal, hence 
it is not finite ext-steady.
\end{example}

As every finitely generated projective module is a direct summand of a finitely generated free module we obtain a consequence of the last
proposition:

\begin{corollary}\label{cp1} A ring $R$ is a right finite ext-steady if and only if every small right ideal is finitely generated.
\end{corollary}

\begin{proof} By Proposition~\ref{p1} it is enough to prove that every small submodule $I$ of every finitely generated free module $R^n$ is finitely
generated. Proceed by induction, the claim is clear for $n=1$ hence suppose that $n>1$ and denote by $\pi:R^n\to R$ the canonical projection
and by $\nu:R\to R^n$ the canonical injection on the first coordinate. As $\pi(I)$ is small so finitely generated submodule of $R$,
and the small module $I+\nu(R)/\nu(R)\cong I/(I\cap\nu(R))$ is embeddable into $R^{n-1}$, the module $I/(I\cap\nu(R))$
is finitely generated by the induction hypothesis. Moreover, $I/\nu\pi(I)$ and $\nu\pi(I)$ are finitely generated as well because
$I\cap\nu(R)\subset \nu\pi(I)$, hence $I$ is finitely generated.
\end{proof}

Recall that a ring is \textsl{right steady} provided every small right module is necessarily finitely generated, \cite{EGT}.
However general structural ring-theoretical characterization of right steady rings is still an open problem, various classes
of rings are known to be right steady (noetherian and perfect \cite{Re}, semiartinian of countable Loewy chain \cite{EGT}).
Let us remark here that the criterion of steadiness for commutative semiartinian rings \cite{RTZ} and for regular semiartinian rings
with primitive factors artinian \cite{Z06} has a similar form as Corollary~\ref{cp1} since steadiness is in this cases equivalent to the
condition that every small ideal of every factor-ring is finitely generated.

\begin{corollary} The category of right modules over right steady or countable ring is right finite ext-steady. 
\end{corollary}

Since there are known countable non-steady rings, as it is illustrated in the following example, the inclusion
of classes of steady rings and finite ext-steady rings is strict.

\begin{example}
Let $F$ be a countable field and $\mathbb X$ is a infinite countable set. Then it is is well known that
over the polynomial ring $F\langle\mathbb X\rangle$ in noncommuting variables $\mathbb X$ every injective module
is small. Thus $F\langle\mathbb X\rangle$ is a non-steady countable ring.
\end{example}

\begin{remark}
From Proposition \ref{p1} and \cite[Example 14]{ZT} we deduce that the ext-steadiness property is not left-right symmetric.  
\end{remark}

We conclude the section with generalization of Corollary~\ref{du>ds} in the case of modules over perfect rings.

\begin{proposition}\label{p2} Let $R$ be a right perfect ring and $M$ a right $R$-module
such that $\Ext^1(M,-)$ commutes with respect to direct sums. Then $M$ is fg-$\Omega^1$.
\end{proposition}

\begin{proof} 
Denote by $J$ the Jacobson radical of $R$. Since $R$ is right perfect, $L/LJ$ is semisimple, hence $L/LJ\cong \bigoplus_{i\in I} S_i$ for a 
family of simple modules $(S_i,i\in I)$. We consider $f:L\to \bigoplus_{i\in I} S_i$
as the canonical projection. Then by Theorem~\ref{oplus} there exists a finite set 
$F\subset I$, and $g:P\to \bigoplus_{i\in I\setminus F} S_i$ such that $\pi_{I\setminus F}f=g \beta$.

Let $\rho:Q\to \bigoplus_{i\in I\setminus F} S_i$ be a projective cover of $\bigoplus_{i\in I\setminus F} S_i$
which exists because $R$ is right perfect. As $\rho$ is surjective, there exists a homomorphisms $\tau:P\to Q$
such that $\rho\tau=g$. 
Note that $\Ker\rho = QJ$ is superfluous in $Q$ and $\rho\tau\beta=g\beta=\pi_{I\setminus F}f$, thus $\tau\beta(L)=Q$ where $Q$ is projective.
Clearly, there exists a homomorphism $\varphi:Q\to P$ such that $\tau\beta\varphi=\id_Q$, hence $L=\varphi(Q)\oplus \Ker \tau\beta$ and
$P=\beta\varphi(Q)\oplus \Ker\tau$. Since the factorization by $\varphi(Q)$ induces a short exact sequence
\[
0\to L/\varphi(Q)\to P/\beta\varphi(Q)\to M\to 0
\]
 and $P/\beta\varphi(Q)$ is projective, it remains to prove that $L/\varphi(Q)$
is finitely generated. This follows from the observations that $V=\Ker \tau\beta\cong L/\varphi(Q)$ and $V/VJ\cong\bigoplus_{i\in F} S_i$.
\end{proof}



Finally we summarize the results about connections between possible commuting properties of a functor $\Ext^1(M,-)$:

\begin{corollary}
For a right $R$-module $M$ we consider the following possible properties:
\begin{enumerate}
 \item[(DS)] $\Ext^1(M,-)$ commutes with respect to direct sums;
  \item[(DU)] $\Ext^1(M,-)$ commutes with respect to direct unions;
   \item[(DL)] $\Ext^1(M,-)$ commutes with respect to direct limits.
\end{enumerate}

Then the following are true:
\begin{enumerate}
 \item If $R$ is hereditary then {\rm (DS)}$\Leftrightarrow${\rm (DU)}$\Leftrightarrow${\rm (DL)}.
  \item $R$ is right coherent if and only if {\rm (DU)}$\Leftrightarrow${\rm (DL)} for all right $R$-modules $M$.
   \item If $R$ is right perfect then {\rm (DS)}$\Leftrightarrow${\rm (DU)}.
\end{enumerate}
\end{corollary}

\begin{proof}
(1) is a consequence of Proposition \ref{equiv-lim} (this is also proved in \cite{Strebel}).

(2) is proved in Theorem \ref{coherent-cor} and Corollary \ref{coherent-1}.

(3) is a consequence of Proposition \ref{p2}.
\end{proof}


 \end{document}